\documentclass[12pt,final]{amsproc}

\usepackage[english]{babel}
\usepackage{latexsym}
\usepackage{amssymb}
\usepackage{amscd}
\usepackage[latin1]{inputenc}
\usepackage{mathrsfs}
\usepackage[notcite,notref]{showkeys}
\usepackage[all]{xy}

\newtheorem{theorem}{Theorem}[section]
\newtheorem{proposition}[theorem]{Proposition}
\newtheorem{corollary}[theorem]{Corollary}
\newtheorem{lemma}[theorem]{Lemma}

\theoremstyle{definition}
\newtheorem{definition}{Definition}

\theoremstyle{remark}

\newcommand{\lra}{\longrightarrow}

\newcommand{\R}{\mathbb{R}}
\newcommand{\N}{\mathbb{N}}
\newcommand{\K}{\mathbb{K}}

\newcommand{\aproof}{\begin{proof}}
\newcommand{\zproof}{\end{proof}}

\def\PB{\operatorname{PB}}

\def\co{\operatorname{co}}

\def\e{\varepsilon}

\def\N{\mathbb{N}}

\def\R{\mathbb R}

\begin{document}

\title[The bounded approximation property and other stories]{On the bounded approximation property on subspaces of $\ell_p$  when $0<p<1$ and related issues}

\author[Cabello, Castillo and Moreno]{F\'elix Cabello S\'anchez, Jes\'us M. F. Castillo and Yolanda Moreno}

\address{Institute of Mathematics (Imuex), Universidad de Extremadura, Avenida de Elvas s/n, 06071 Badajoz, Spain}
             \email{fcabello@unex.es}
             \email{castillo@unex.es}



\address{Institute of Mathematics (Imuex), Escuela Polit\'ecnica, Universidad de Extremadura, Avenida de la Universidad s/n, 10071 C\'aceres, Spain}
             \email{ymoreno@unex.es}



\thanks{Supported in part by DGICYT project MTM2016$\cdot$76958$\cdot$C2$\cdot$1$\cdot$P (Spain) and Junta de Extremadura project IB$\cdot$16056.}


\bigskip

\bigskip

\maketitle

\begin{abstract} This paper studies the bounded approximation property (BAP) in quasi Banach spaces. in the first part of the paper we show that the kernel of any surjective operator $\ell_p\lra X$ has the BAP when $X$ has it and $0<p\leq 1$, which is an analogue of the corresponding result of Lusky for Banach spaces. We then obtain and study nonlocally convex versions of the Kadec-Pe\l czy\'nski-Wojtaszczyk complementably universal spaces for Banach spaces with the BAP.
\end{abstract}

\section{Introduction}
While the Bounded Approximation Property (BAP) has received extensive attention in Banach spaces, things are different in quasi Banach spaces, and only a few papers touch the topic \cite{kaltuni, k-loc}. In this paper we study two different aspects of the BAP in quasi Banach spaces by means of a mixture of linear and nonlinear techniques. The first one is a study of the BAP on subspaces of $\ell_p$ when $0<p\leq 1$, which continues the circle of ideas discussed in \cite{L, CM-Isr, FJP, FJ} in the more arid ground of quasi Banach spaces. Our main result in this line is an analogue of Lusky's theorem \cite{L} holds: if $Q:\ell_p\lra X$ is a quotient map and $X$ has the BAP, then so $\ker Q$ has it. The case $p=1$ is, as we have said, due to Lusky, who proved that $\ker Q$ has a basis if $X$ has it. The proof in \cite{L} is based on the duality between the $\mathscr L_1$-spaces and  the $\mathscr L_\infty$-spaces, a path that cannot be followed when $p<1$ for obvious reasons. Our exposition not only shows that the result comfortably survives to the lack of local convexity, but also that the reason behind it is that $\ell_p$ is projective in the category of $p$-Banach spaces. The leading idea follows \cite{CM-Isr} in showing that the kernel of a quotient map $\ell_p\lra X$ is very similar to a certain ``nonlinear envelope'' of $X$, which inherits the BAP from $X$.\medskip

The second topic of the paper is the study of the ``largest'' separable $p$-Banach space having the BAP.
The existence of a ``complementably universal'' $p$-Banach space for the BAP is mentioned by Kalton in \cite{kaltuni} just adding that ``it is easy to duplicate the results for Banach spaces...'' It is clear from  \cite[Theorem 4.1(b) and Corollary 7.2]{kaltuni} that Kalton is referring to Pe\l czy\'nski papers \cite{pelcuni,pelcbap}. Alternative approaches could be to construct a $p$-convex version of Kadec' space in \cite{kade} or work out the $p$-convex version of Pe\l czy\'nski--Wojtaszczyk space in \cite{p-w}.\medskip

Here, by following ideas of Garbuli\'nska-W\c egrzyn, we show that these spaces arise quite naturally as ``approximate Fra\"\i ss\'e limits'' in a certain category whose ``morphisms'' are pairs of isometries and contractive projections acting between finite dimensional spaces. These techniques, promoted in the study of Banach spaces by Kubi\'s \cite{kubis, kukubis}, shed light on the isometric structure of Kalton and Kadec spaces. In particular we show that $\mathscr K_p$ can be characterized as the only $p$-Banach space with a monotone finite dimensional decomposition which has a certain extension property that we have called (we apologize in advance) ``almost universal complemented disposition''. In the end, it turns out that, within the ``category of pairs'', these spaces play the same role than Gurariy spaces in the ``isometric category'', though their structural properties are completely different.

\section*{Notations, conventions}
We consider real or complex linear spaces and we denote by $\K$ the scalar field.

\subsection*{Quasinorms and $p$-norms}
General references for quasi Banach spaces are \cite{KPR, kaltonHB}. Here we only recall some definitions, mainly to fix the notation.

A quasinorm on a linear space $X$ is a map $\|\cdot\|:X\lra \R$ with the following properties:
\begin{itemize}
\item $\|x\|=0$ if and only if $x=0$;
\item $\|\lambda x\|=|\lambda|\|x\|$ for every $\lambda\in\K$ and every $x\in X$.
\item $\|x+y\|\leq \Delta(\|x\|+\|y\|)$ for some constant $\Delta$ and every $x,y\in X$.
\end{itemize}
A quasinorm is called a $p$-norm (where $0<p\leq 1$) if, in addition, is $p$-subadditive:
\begin{itemize}
\item  $\|x+y\|^p\leq \|x\|^p+\|y\|^p$ for every $x,y\in X$.
\end{itemize}
A quasinorm gives rise to a topology on $X$, namely the coarsest linear topology for with the ball $B=\{x\in X:\|x\|\leq 1\}$ is a neighbourhood of the origin. The resulting linear topological space is called a quasinormed space and a quasi Banach space if it is complete. If the quasinorm happens to be a $p$-norm we speak of $p$-normed and $p$-Banach spaces. An important result due to Aoki and Rolewicz states that every quasinorm is equivalent to some $p$-norm, in the sense that they induce the same topology on the underlying space.

\subsection*{Operators} By an operator we mean a linear and continuous (equivalently, bounded) map acting between quasinormed spaces. If $T:X\lra Y$ is an operator, then
the quasinorm of $T$ is the number $\|T\|=\sup_{\|x\|\leq 1}\|Tx\|$. The space of all operators from $X$ to $Y$ is denoted by $L(X,Y)$. This should not be confused with the space of all linear maps, bounded or not, which is denoted by ${\sf L}(X,Y)$.

An operator $T$ is said to be contractive (or a contraction) if $\|T\|\leq 1$ and an isometry if it preserves the involved quasinorms: $\|Tx\|=\|x\|$. If
$$
(1+\e)^{-1}\|x\|\leq \|Tx\|\leq (1+\e)\|x\|
$$
for every $x$ in the domain of $T$, then $T$ is called an $\e$-isometry. Note that most authors use $1-\e$ instead of $
(1+\e)^{-1}$ for the lower bound.

Neither isometries nor $\e$-isometries are assumed to be surjective. However, we say that two spaces are isometric (or $\e$-isometric) if and only if there is a surjective isometry ($\e$-isometry) between them.

\subsection*{Bounded Approximation Property}

Since there are a number of equivalent definitions of the BAP let us record right now the version that we shall use along the paper.

\begin{definition}
A quasi Banach space $X$ has the $\lambda$-approximation property ($\lambda$-AP, for short) if for every finite-dimensional subspace $E$ and every $\e>0$ there is a finite-rank operator $T\in L(X)$ such that $Tx=x$ for all $x\in E$ and $\|T\|\leq \lambda+\e$. (Equivalently one can require that $\|x-Tx\|\leq\e\|x\|$ for all $x\in E$ and $\|T\|\leq\lambda$.)

A quasi Banach space has the BAP if it has the $\lambda$-AP for some $\lambda\geq1$.
\end{definition}

As the reader may guess, the lack of local convexity tends
to interfere with the BAP and so quasi Banach spaces with the BAP are more scarce than their Banach relatives. This must be understood in a ``probabilistic sense'' since each Banach space is also a quasi Banach space.

Nevertheless many natural examples of nonlocally convex quasi Banach spaces
have the BAP, among them the sequence spaces $\ell_p$, the Hardy classes $H^p$ and the Schatten classes $S_p$ for $0<p<1$.

Other natural spaces such as Lebesgue's $L_p$ and the $p$-Gurariy spaces $\mathscr G_p$  for $0<p<1$ lack it since they have trivial dual and no finite rank endomorphism, apart from zero.

\section{BAP on subspaces of $\ell_p$}


\subsection{Exact sequences of quasi Banach spaces}
An exact sequence of quasi Banach spaces is a diagram of quasi Banach spaces and (linear, continuous) operators
\begin{equation}\label{yzx}
\begin{CD}
0@>>> Y@>\imath >> Z@> \pi >> X@>>> 0
\end{CD}
\end{equation}
in which the kernel of each operator agrees with the image of the preceding one: in particular $\ker\pi=\imath[Y]$.

For our present purposes the most important examples of exact sequences are obtained considering a separable $p$-Banach space $X$, taking a quotient map $\pi:\ell_p\lra X$ and letting
$$
\begin{CD}
0@>>> Y@>\imath >> \ell_p@> \pi >> X@>>> 0,
\end{CD}
$$
where $Y=\ker\pi $ and $\imath$ is the inclusion map.

As oftens happens in Mathematics, having defined a class of objects, one has to decide when two objects are ``essentially the same''. The canonical definition for exact sequences reads as follows. Two exact sequences $0\lra Y\lra Z_i\lra X\lra 0$ are equivalent if there is an operator $u$ making commutative the diagram
$$
\xymatrixcolsep{3.5pc}\xymatrixrowsep{0.5pc}\xymatrix{
  & & Z_1 \ar[dd]^u \ar[dr] \\
 0 \ar[r] & Y \ar[ur]\ar[dr] & & X \ar[r] & 0\\
&  & Z_2 \ar[ur]
}
$$

\subsection{Exact sequences of $p$-Banach spaces and $p$-linear maps}
We now define quasilinear maps, which will play a fundamental role in our study of the BAP. A homogeneous map $\Phi:X\lra Y$ acting between quasi Banach spaces is said to be quasilinear if it obeys an estimate
$$
\|\Phi(x+x')-\Phi(x)-\Phi(x')\|\leq Q(\|x\|+\|x'\|)
$$
for some constant $Q$ an every $x,x'\in X$. These maps are the nonlinear side of short exact sequences. Indeed if one has  a quasilinear map  $\Phi:X\lra Y$ one can create an extension of $X$ by $Y$ as follows:

\begin{itemize}
\item Endow the product space $Y\times X$ with the quasinorm $\|(y,x)\|_\Phi=\|y-\Phi(x)\|+\|x\|$ and denote the resulting quasinormed space by $Y\oplus_\Phi X$.
\item Set up the exact sequence $$
\begin{CD}
0@>>> Y@>\imath >> Y\oplus_\Phi X @> \pi >> X@>>> 0,
\end{CD}
$$
with $\imath(y)=(y,0)$ and $\pi(y,x)=x$.

\item Note that $\imath$ is an isometry (it preserves the quasinorms) and that $\pi$ maps the unit ball of $Y\oplus_\Phi X$ onto that of $X$. In particular both $\imath$ and $\pi$ are relatively open operators and this implies that $Y\oplus_\Phi X$ is complete, that is, a quasi Banach space.
\end{itemize}

Thus, quasilinear maps give rise to extensions of quasi Banach spaces. When do two quasilinear maps $\Phi$ and $\Psi$ induce equivalent extensions? Precisely when the difference $\Phi- \Psi$ can be written as the sum of a bounded and a linear map (both from $X$ to $Y$).

Which extensions arise from quasilinear maps? All, up to equivalence.
We do not need these facts, that the reader can see in Kalton-Peck \cite{KP}.
This paper also contains one of the earlier and most spectacular applications of quasilinear maps to the study of exact sequences, namely that the classes of $p$-Banach spaces are not closed under extensions: for each $0<p\leq 1$ there is an exact sequence (\ref{yzx}) in which both $X$ and $Y$ are $p$-Banach spaces but $Z$ is not isomorphic to a $p$-Banach space.

Thus the following question arises. If $\Phi:X\lra Y$ is a quasilinear map acting between $p$-Banach spaces, when does the induced extension live in the category of $p$-Banach spaces? The answer is quite easy one one realizes two facts: first that a quasinorm may be equivalent to a $p$-norm without being itself a $p$-norm, and second, that a quasinorm on $Z$ is equivalent to a $p$-norm if and only if there is a constant $C$ such that
$$
\left\|\sum_i z_i\right\|\leq C\left(\sum_i \| z_i\|^p\right)^{1/p}
$$
for every $n\in\N$ and  every  $z_1,\ldots, z_n\in Z$.

\begin{lemma}
\label{convex}
Let $\Phi:X\lra Y$ be a quasilinear map acting
between $p$-normed spaces. Then $Y\oplus_\Phi Z$ is  isomorphic to a $p$-normed space if and only
there is a constant $K$ such that for every $n\in\N$ and  every  $x_1,\ldots, x_n\in X$ one has
 \begin{equation}\label{zestimate}
\left\|\Phi\left(\sum_{i=1}^nx_i\right)-\sum_{i=1}^n\Phi(x_i)\right\|
\leq K \; \left(\sum_{i=1}^n\|x_i\|^p\right)^{1/p}.\end{equation}
\end{lemma}

\begin{proof}
If $\Phi$ satisfies (\ref{zestimate}), then, recalling that $s^p+t^p\leq 2^{1-p}(s+t)^p$ for $0<p\leq 1$,
$$
\begin{aligned}
\left\|\sum_i (y_i,x_i)\right\|_\Phi &= \left\|\sum_i
y_i-\Phi\left(\sum_ix_i\right)\right\|+\left\|\sum_ix_i\right\|\\
&\leq \left(\left\|\sum_i y_i-\sum_i\Phi x_i\right\|^p+
\left\|\Phi\left(\sum_ix_i\right)-\sum_i\Phi x_i\right\|^p
+\left\|\sum_ix_i\right\|^p\right)^{1/p}\\
&\leq \left(\sum_i \left(\| y_i-\Phi x_i\|^p+\|x_i\|^p\right)
+K^p \sum_i\|x_i\|^p\right)^{1/p}\\
&\leq (1+K^p)^{1/p}  \left(\sum_i \left(\| y_i-\Phi x_i\|^p+\|x_i\|^p\right)\right)^{1/p}\\
&\leq (1+K^p)^{1/p}  \left(\sum_i 2^{1-p}\big{(}\| y_i-\Phi x_i\|+\|x_i\|\big{)}^p\right)^{1/p}\\
&\leq (1+K^p)^{1/p} 2^{1/p-1} \left(\sum_i  \|(y_i,x_i)\|_\Phi^p\right)^{1/p}.
\end{aligned}
$$
Conversely, if $Y\oplus_\Phi X$ is locally
$p$-convex, then for finitely many $x_i\in X$ one has
$$
\left\|\sum_i(\Phi(x_i),x_i)\right\|_\Phi\leq M\left(\sum_i\|(\Phi(x_i), x_i)\|^p\right)^{1/p},
$$
hence
$$
\left\|\sum_i \Phi x_i-\Phi\left(\sum_i x_i\right)\right\|\leq M\left(\sum_i\|x_i\|^p\right)^{1/p}.\vspace{-15pt}
$$
\end{proof}

\begin{definition}
A homogeneous mapping  $\Phi: X \lra Y$ satisfying the estimate (\ref{zestimate}) is said to be $p$-linear. The least possible constant $K$ for which  (\ref{zestimate}) holds shall be referred to as the $p$-linearity constant of $\Phi$ and denoted by $Q^{(p)}(\Phi)$.
\end{definition}

This can be seen as a stronger form of
quasilinearity which involves an arbitrary number of variables
(instead of two).

Thus, to each couple of  quasi Banach spaces $X$ and $Y$ we will attach the space of all $p$-linear maps $\Phi:X\lra Y$ that we denote by
${\sf Q}^{(p)}(X, Y)$. We measure size in this space just using the $p$-linearity constant. It is clear that the asignment $\Phi\longmapsto Q^{(p)}(\Phi)$ has the properties of a $p$-norm except for the fact that $Q^{(p)}(\Phi)=0$ does not imply that $\Phi=0$: actually it is clear that $Q^{(p)}(\Phi)=0$ if and only if $\Phi$ is linear, not necessarily bounded.

There are two obvious ways to obtain a genuine $p$-normed space from ${\sf Q}^{(p)}(X, Y)$: the most natural one is to pass to the quotient over ${\sf L}(X,Y)$, the subspace of  linear maps, bounded or not. This has the disadvantage that the elements of the quotient are not longer ``mappings''. To avoid this (formal) problem let us fix once and for all a Hamel basis $H$ of $X$ and take $\tilde{\sf Q}^{(p)}(X, Y)$ as the subspace of
${\sf Q}^{(p)}(X, Y)$ of those $p$-linear mappings vanishing on $H$.

Observe that for every $p$-linear $\Phi: X\lra Y$ there is exactly one
linear map that agrees with $\Phi$ on $H$, namely
$$
L(x)=\sum_{h\in H}\lambda_h(x)\Phi(h),
$$
where $x=\sum_{h\in H}\lambda_h(x)h$ is the expansion of $x$ with respect to $H$.
The map $\tilde{\Phi}=\Phi-L$ belongs to $\tilde{\sf Q}^{(p)}(X, Y)$ and one has
$ {\sf Q}^{(p)}(X, Y)= \tilde{\sf Q}^{(p)}(X, Y)\oplus {\sf L}(X,Y)$.

\subsection{Linearization of $p$-linear maps}

Our immediate objective is to get a $p$-Banach space that linearizes $p$-linear maps defined on a given $p$-Banach space $X$. That is we need to attach to $X$ a certain space $\clubsuit$ so that each $p$-linear map $\Phi:X\lra Y$ corresponds to a bounded operator $T:\clubsuit\lra Y$ whichever it is the $p$-Banach space $Y$.

In \cite{CM-Isr}, where only Banach spaces and 1-linear maps are considered, this is achieved
by means of certain ``natural'' predual of the space of 1-linear forms $\tilde{\sf Q}^{(1)}(X, \K)$. This approach makes no sense in our current circunstances because there are lots of $p$-Banach spaces where the only bounded linear functional is the zero map.

So, the first thing we need is a $p$-Banach space $\mathscr U$ such that $L(Y,\mathscr U)\neq 0$ for all $p$-Banach spaces $Y$ with ``controlled cardinalities''.

Such an object can be obtained in various ways, but for our present purposes the most simple choice is that coming from the following remark.

\begin{lemma}
For every cardinal $\kappa$ there is a $p$-Banach space $\mathscr U_\kappa$ containing an isometric copy of every $p$-Banach space whose cardinality (not density character) is $\kappa$ or less.
\end{lemma}

\begin{proof}
Let $I$ be the set of all closed subspaces of $\ell_p^\kappa$ and for each $Y\in I$ consider the quotient $\ell_p^\kappa/Y$. Set $\mathscr U_\kappa=\ell_p(I, \ell_p^\kappa/Y)$.
\end{proof}

Another possibility is to take the so-called $p$-Gurariy space $\mathscr G_p$ from \cite{CGK} or a countable ultrapower of the space $\mathscr K_p$ that appears in Section~\ref{sec:aucd}. The only property of these spaces that we need is universality with respect to the class of separable $p$-Banach spaces.

So, let $\mathscr U$ be any $p$-Banach space containing an isometric copy of every separable $p$-Banach space. Let us replace the ground field by $\mathscr U$ and see what happens.

First, we consider the space of $p$-linear maps
${\sf Q}^{(p)}(X, \mathscr U)$ and the subspace $\tilde{\sf Q}^{(p)}(X, \mathscr U)$.

Then $\co^{(p)}(X)$ is the closed subspace of $L\left( \tilde{\sf Q}^{(p)}(X, \mathscr U), \mathscr U	\right)$ generated by the evaluation maps $\delta_x$, with $x\in X$.
Of course we must check that $\delta_x$ is bounded, but if $\Phi$ vanishes on $H$, then, writing $x=\sum_h\lambda_h(x)h$ we have
$$
\|\delta_x(\Phi)\|= \left\|\Phi(x)- \sum_h\lambda_h(x)\Phi(h)\right\|\leq Q^{(p)}(\Phi) \left(\sum_h|\lambda_h(x)|^p\|h\|^p\right)^{1/p}.
$$

The universal property behind this construction  appears now:

\begin{theorem}\label{th:univ}
The map $\mho: X\lra \co^{(p)}(X)$ given by $\mho(x)=\delta_x$ is $p$-linear, with $Q^{(p)}(\mho)\leq 1$. For every separable $p$-Banach space $Y$ and every $p$-linear map $\Phi:X\lra Y$ vanishing on $H$ there is a unique operator $T: \co^{(p)}(X)\lra Y$ such that $\Phi=T\circ\mho$ and, moreover, one has $\|T\|=Q^{(p)}(\Phi)$.
\end{theorem}

\begin{proof}
The only reason for which the Theorem needs a proof is that the definition of
$\co^{(p)}(X)$ is a kind of tongue-twister.

That $\mho$ is homogeneous is trivial: take $\lambda\in \K, x\in X$ and $\Phi\in
 \tilde{\sf Q}^{(p)}(X, \mathscr U)$. Then
 $$
 \mho(\lambda x)(\Phi)=\delta_{\lambda x}\Phi= \Phi(\lambda x)=
 \lambda\Phi(x)=  (\lambda\mho x)(\Phi),
 $$
 so $ \mho(\lambda x)=  \lambda\mho(x)$.

Now, take finitely many $x_i\in X$ and set $x=\sum_i x_i$. One has
$$
\begin{aligned}
\left\| \mho(x)-\sum_i\mho(x_i)\right\|&= \sup_{Q^{(p)}(\Phi)\leq 1} \left\| \left(\mho(x)-\sum_i\mho(x_i)\right)(\Phi)\right\|\\
&=
\sup_{Q^{(p)}(\Phi)\leq 1} \left\| \Phi(x)-\sum_i\Phi(x_i)\right\|\leq \left(\sum_i\|x_i\|^p\right)^{1/p},
\end{aligned}
$$
 hence $Q^{(p)}(\mho)\leq 1$.

To prove the second statement let us first consider the case where $Y=\mathscr U$. So, let us assume $\Phi: X\lra \mathscr U$ is a (fixed) $p$-linear map vanishing on $H$. The required operator $T$ is just the restriction to $\co^{(p)}(X)$ of the ``evaluation at $\Phi$'':
$$
\delta_\Phi: L\left( \tilde{\sf Q}^{(p)}(X, \mathscr U), \mathscr U	\right) \lra \mathscr U,
$$
that, is, $\delta_\Phi(u)=u(\Phi)$. By the very definition of the operator $p$-norm one has $\|\delta_\Phi\|=\sup_{\|u\|\leq 1}\|u(\Phi)\|\leq Q^{(p)}(\Phi)$. Since $\co^{(p)}(X)$ is a subspace of $ L\left( \tilde{\sf Q}^{(p)}(X, \mathscr U), \mathscr U	\right)$ one also has $\|T\|\leq \|\delta_\Phi\|\leq Q^{(p)}(\Phi)$.

The identity $\Phi=T\circ\mho$ is trivial too: take $x\in X$; then
$$
T(\mho(x))=T(\delta_x)=\delta_\Phi(\delta_x)=\delta_x(\Phi)=\Phi(x).
$$
This also shows that $\|T\|$ and $Q^{(p)}(\Phi)$ actually agree since $\Phi=T\circ\mho$ implies that  $Q^{(p)}(\Phi)\leq \|T\|Q^{(p)}(\mho)\leq \|T\|$.

To complete the proof, suppose $\Phi: X\lra Y$ is $p$-linear and vanishes on $H$. Fixing an isometry $\kappa: Y\lra \mathscr U$ we can consider the composition $\kappa\circ \Phi:X\lra Y\lra\mathscr U$ to get an operator $T: \co^{(p)}(X)\lra \mathscr U$ such that $T\circ\mho=\kappa\circ\Phi$:
$$
\xymatrix{
X \ar[d]_\Phi \ar[rr]^\mho & &  \co^{(p)}(X) \ar[d]^T\\
Y\ar[rr]^\kappa & & \mathscr U
}
$$
But $T$, being bounded, takes values in $\kappa[Y]$, since $T(\delta_x)$ falls there for each $x\in X$ and the subspace  spanned by the evaluation functionals is dense in $\co^{(p)}(X)$. Now, the operator $\kappa^{-1}\circ T$ does the trick.
\end{proof}

And obvious consequence of the preceding Theorem is that there is a ``universal'' exact sequence
\begin{equation}\label{univ:seq}
\begin{CD}
0@>>> \co^{(p)}(X) @>\imath>> \co^{(p)}(X)\oplus_\mho X @>\pi>> X@>>> 0.
\end{CD}
\end{equation}
This sequence will play a role later.

\subsection{Tuning $p$-linear maps}

We now arrive at the key point of our argumentation. We want to prove that if $X$ has the BAP, then so  $\co^{(p)}(X)$ does, a fact involving (finite rank) operators on  $\co^{(p)}(X)$. But operators on  $\co^{(p)}(X)$ correspond to $p$-linear maps $X\lra  \co^{(p)}(X)$ so, at the end of the day, our statement refers to the behaviour of certain quasilinear maps defined on $X$.

\begin{lemma}\label{lem:bestsuited}
Let $\Phi:F\lra Y$ be a $p$-linear map, where $F$ and $Y$ are $p$-normed spaces, with $F$ finite-dimensional. Then, for each $\e>0$, there is a homogeneous mapping  $\Phi':F\lra Y$ such that:
\begin{itemize}
\item $\|\Phi-\Phi'\|\leq (1+\e)Q^{(p)}(\Phi)$,
\item $\Phi'[F]$ spans a finite dimensional subspace of $Y$,
\item $Q^{(p)}(\Phi')\leq  3^{1/p}(1+\e)Q^{(p)}(\Phi)$.
\end{itemize}
\end{lemma}

\begin{proof}
The main step of the proof is to associate to every $f\in F$ a good ``$p$-convex'' decomposition in a judicious way.

Let $S$ be the unit sphere of $F$. As $S$ is compact, for fixed $\delta>0$ we can find $f_1,\dots, f_n\in S$ such that, for every $f\in S$ there is $f_i$ such that $\|f-f_i\|<\delta$.

Pick any $f\in S$ and take $1\leq i\leq n$ such that
$$
\|f-f_i\|=\min_{1\leq k\leq n} \|f-f_k\|<\delta
$$
(take the least index if the minimun is attained on two or more vectors).
Let us set $f_0=0$ and consider the new point
$g=(f-f_i)/\|f-f_i\|$. Now, take $0\leq j\leq n$ minimizing $\|g-f_j\|$, with the same tie-break rule as before. We have
$$
\left\|f-f_i-(\|f-f_i\|) f_j\right\|< \delta^2.
$$
If $f=f_i$, then we take $j=0$.
In any case, continuing in this way we can select sequences $i(k)$ and $\lambda_k$ such that:
\begin{itemize}
\item $0\leq i(k)\leq n$ and $0\leq\lambda_k<\delta^k$ for all $k=0,1,2,\dots$.
\item For every $m\in \N$ one has $\|f-\sum_{0\leq k\leq m}\lambda_kf_{i(k)}\|\leq \delta^m$.
\end{itemize}
Grouping the terms in the obvious way and taking into account that the $\ell_1$-norm is dominated by the $\ell_p$ quasinorm, we can write
$$
f=\sum_{i=1}^n c_if_i\quad\text{with}\quad \left(\sum_{i=1}^n |c_i|^p\right)^{1/p}\leq  \left(\sum_{k=0}^\infty  \lambda_k^p\right)^{1/p}\leq \left(\frac{1}{1-\delta^{p}}\right)^{1/p}<1+\e
$$
 for $\delta$ sufficiently small.

This ``greedy algorithm'' especifies a unique decomposition for each $f$ in the unit sphere of $F$. However it does not guarantee any kind of homogeneity in these decompositions.

To amend this, let $S_0$ be a maximal subset of the unit sphere of $F$ with the property that any two points of $S_0$ are linearly independent (of course, when the ground field is $\R$ this just means that $S_0$ does not contain ``antipodal'' points). Equivalently, $S_0$ is a subset of the sphere such that every nonzero $f\in F$ can be written in a unique way as $f=cx$, with $c\in\mathbb K$ and $x\in S_0$.

Now we define $\Phi':F\lra Y$ as follows: if $f\in S_0$, then we put
$$
\Phi'(f)= \sum_{i=1}^n c_i\Phi(f_i)
$$
where $f=\sum_{i=1}^n c_if_i$ is the decomposition provided by the algorithm. We extend the map to the whole of $F$ by homogeneity: that is, for arbitrary $f\in F$ we write $x=cf$, with $c\in \mathbb K$ and $f\in S_0$, in the only way that this can be done, and we set $\Phi'(f)=c\Phi'(f)$.

It is clear that the resulting map is homogeneous. Let us check that $\Phi'$ works properly. Let  $K$ denote the $p$-linearity constant of the starting map $\Phi$. For $f\in S_0$, one has
$$
\|\Phi(f)-\Phi'(f)\|= \left\|\Phi(f)-  \sum_{i=1}^n c_i\Phi(f_i) \right\|\leq Q^{(p)}  \left(\sum_{i=1}^n |c_i|^p\right)^{1/p}< (1+\e)K,
$$
and for arbitrary $f$ just use the homogeneity of both maps.

It is obvious that the range of $\Phi'$ is contained in $[\Phi(f_1),\dots,\Phi(f_n)]$, which is a finite dimensional subspace of $Y$. Finally, to estimate the $p$-linearity constant of  $\Phi'$ we have, for $x_1,\dots,x_k\in F$, letting $x=\sum_{1\leq i\leq k}x_i$,
$$\begin{aligned}
\left\|\Phi'(x)-\sum_{i=1}^k\Phi'(x_i)\right\|^p
&= \left\|\Phi'(x)-
\Phi(x)
+\Phi(x)- \sum_{i=1}^k\Phi x_i
+ \sum_{i=1}^k\Phi x_i
-
\sum_{i=1}^k\Phi' x_i \right\|^p\\
& \leq
\left( (1+\e)K\right)^{p}\|x\|^p+
K^{p} \sum_{i=1}^k\|x_i\|^p+
\left( (1+\e)K\right)^{p} \sum_{i=1}^k\|x_i\|^p\\
& \leq 3(1+\e)^p K^{p} \sum_{i=1}^k\|x_i\|^p,
\end{aligned}$$
hence $Q^{(p)}(\Phi')\leq 3^{1/p}(1+\e) Q^{(p)}(\Phi)$.
\end{proof}

\begin{lemma}\label{con x}
Let $\Phi:X\lra Y$ be a $p$-linear map, where $X$ and $Y$ are $p$-normed spaces. Let $F$ be a finite-dimensional subspace of $X$ and let $x_1,\dots,x_m$ be points in the unit sphere of $F$. Then, for each $\e>0$, there is a homogeneous mapping  $\Phi_F:X\lra Y$ such that:
\begin{itemize}
\item $\|\Phi-\Phi_F\|\leq (1+\e)Q^{(p)}(\Phi)$,
\item $\Phi_F[F]$ spans a finite dimensional subspace of $Y$,
\item $\Phi_F(x_i)=\Phi(x_i)$ for $1\leq i\leq m$.
\item $Q^{(p)}(\Phi_F)\leq  3^{1/p}(1+\e)Q^{(p)}(\Phi)$.
\end{itemize}
\end{lemma}

\begin{proof}
Fix $\e>0$ and let us consider the map $\Phi|_F$ as a $p$-linear map from $F$ to $Y$. Let us apply Lemma~\ref{lem:bestsuited} to $\Phi|_F$. Be sure that the $\delta$-net ${f_1,\dots,f_n}$ appearing in the fourth line of the proof contains the set $\{x_1,\dots,x_m\}$ adding these points if necessary and observe that the outcoming map $\Phi':F\lra Y$ agrees with $\Phi$ at every $f_i$. We define $\Phi_F:X\lra Y$ taking
$$
\Phi_F(x)=\begin{cases}\Phi'(x)&\text{if $x\in F$}\\
\Phi(x)&\text{otherwise}\end{cases}
$$
Now, observe that  $\|\Phi-\Phi_F\|\leq (1+\e)Q^{(p)}(\Phi)$ and repeat the proof of Lemma~\ref{lem:bestsuited} with $\Phi_F$ replacing $\Phi'$.
\end{proof}

\subsection{Transferring the BAP from $X$ to $\co^{(p)}(X)$}

\begin{proposition}\label{eleuno} If $X$ is a $p$-Banach space with the $\lambda$-AP, then $\co^{(p)}(X)$ has
the $3^{1/p}\lambda$-AP.
\end{proposition}

\begin{proof} Let
$\mho: X\lra \co^{(p)}(X)$ be the universal $p$-linear map appearing in Theorem~\ref{th:univ}.

Let $F$ be a finite dimensional
subspace of $\co^{(p)}(X)$ and pick $\e>0$. We can assume without loss of generality
that $F = [\mho (x_1), \dots, \mho (x_m)]$, where $x_j\in X$ for $1\leq j\leq m$. Let $H_0$ be a finite subset of the Hamel basis $H$ whose linear span contains $[x_1, \dots, x_m]$. Now, let $S$ be a finite rank operator on $X$ fixing $H_0$, with $\|S\|\leq \lambda+\e$.\medskip

Since $Q^{(p)}(\mho)=1$, applying Lemma~\ref{con x} we can construct a small perturbation $\mho':X \lra \co^{(p)}X$ satisfying:
\begin{itemize}\setlength\itemsep{-0.2em}
\item $\mho'(x_j)=\mho(x_j)$ for $1\leq j\leq m$,
\item $\mho'(h)=\mho(h)=0$ for $h\in H_0$,
\item $\|\mho'-\mho\|\leq 1+\e$,
\item $Q^{(p)}(\mho')\leq 3^{1/p}(1+\e)$ and
\item $\mho' \left( S[X] \right) $ spans a finite dimensional subspace of $\co^{(p)}X$.
\end{itemize}

Let $\Phi: X\lra \co^{(p)}X$ be the ``version'' of $\mho'\circ S$ that vanishes on $H$, that is,
$$
\Phi(x)=\mho'(S(x))- \sum_{h\in H}\lambda_h(x) \mho'(S(h)).
$$
The universal property of $\mho$ yields an operator $T:\co^{(p)}(X)\lra\co^{(p)}(X)$ such that
$
T\circ \mho= \Phi
$.
Let us check that $T$ has the required properties:
\begin{itemize}\setlength\itemsep{-0.1em}
\item $\|T\|= Q^{(p)}(\Phi)= Q^{(1)}(\mho'\circ S)\leq
Q^{(p)}(\mho')\|S\|\leq 3^{1/p}(1+\e)(\lambda+\e)$.
\item $T$ has finite rank: the image of $\mho'\circ S$ spans a finite dimensional subspace of $\co^{(p)}(X)$, say $E$, and therefore one also has $\Phi[X]\subset E$. Thus, $T(\delta_x)\in E$ for all $x\in X$  and since $\co^{(p)}(X)$ is the closure of the space spanned by the points of the form $\delta_x$ and $T$ is continuous we see that $T[\co^{(p)}X]\subset E$.
\item $T$ fixes $F$. Indeed, since $x_j=\sum_{h\in H_0}\lambda_h(x_j)h$ for $1\leq j\leq m$, one has\end{itemize}
$$
T(\mho(x_j))= \Phi(x_j)= \mho'(Sx_j)-\sum_{h\in H_0}\lambda_h(x_j)\mho'(Sh)
=
 \mho'(x_j)-\sum_{h\in H_0}\lambda_h(x_j)\mho'(h)=  \mho(x_j).
$$
 This completes the proof.
\end{proof}

\subsection{Projective presentations vs. $\co^{(p)}(X)\oplus_\mho X$}

\begin{proposition}\label{prop:vscop}
For each exact sequence of $p$-Banach spaces
$$
\begin{CD}
0@>>> Y@>>> Z@>>> X@>>> 0
\end{CD}
$$
there are operators $T$ and $S$ making commutative the diagram
\begin{equation}\label{isobvious}
\xymatrixcolsep{3.5pc}\xymatrix{
0\ar[r] & \co^{(p)}X\ar[r]^-\imath \ar[d]_T & \co^{(p)}\!X\!\oplus_\mho\! X \ar[r]^-\pi \ar[d]_S & X\ar[r]\ar@{=}[d] & 0\\
0\ar[r] & Y\ar[r]^\jmath  &Z \ar[r]^\varpi & X\ar[r] & 0
}
\end{equation}
\end{proposition}

\begin{proof}
There is no loss of generality in assuming that $Y=\ker\varpi$ and $\jmath$ is the inclusion mapping. Let $B:X\lra Z$ be a homogeneous bounded section of the quotient map, that is, $\varpi\circ B={\bf 1}_X$ and let $L:X\lra Z$ be the only linear map that agrees with $B$ on $H$, that is,
$$
L(x)=\sum_{h\in H}\lambda_h(x)B(h).
$$
Then $\Phi=B-L$ takes values in $Y$ and is $p$-linear, with $Q^{(p)}(\Phi)= Q^{(p)}(B)\leq 3^{1/p}\|B\|$, so there is an operator $T:\co^{(p)}(X)\lra Y$ such that $\Phi=T\circ\mho$.

Define $S(u,x)=Tu+Lx$ and let us check that $S$ is bounded -- the commutativity of (\ref{isobvious}) is obvious. Note that $L=B-\Phi$, hence
$$
\begin{aligned}
\|S(u,x)\|&=\|Tu+Lx\|=\|Tu-\Phi(x)+B(x)\|\\
&=2^{1/p}\left(\|Tu-T\mho(x)\|+\|B\|\|x\|\right)\leq 2^{1/p}\max\left(\|T\|,\|B\|\right)\|(u,x)\|_\mho,
\end{aligned}
$$
as required.
\end{proof}

\subsection{The diagonal sequence}
Regarding the diagram appearing in Proposition~\ref{prop:vscop} we must add the following:

\begin{lemma}\label{lem:diagonal}
Suppose we are given a commutative diagram of quasi Banach spaces and operators
$$
\begin{CD}
0@>>> A @>\imath>> B @>>> C @>>> 0\\
 & & @VTVV @VSVV @| \\
0@>>> D @>\jmath>> E @>>> C @>>> 0
\end{CD}
$$
with exact rows. Then the sequence
$$
\begin{CD}
0@>>> A @>\imath\times(-T)>> B\times D @>S\oplus\jmath>> E @>>> 0
\end{CD}
$$
where $(\imath\times(-T))(a)=(\imath(a),-T(a))$ and $(S\oplus\jmath)(b,d)=Sb+\jmath(d)$  is exact.
\end{lemma}

We will not spoil the reader's fun writing a proof down.

\subsection{The end of all this}

\begin{corollary}\label{cor:theend}
Let $Y$ be  a closed subspace of $\ell_p$. If $\ell_p/Y$ has the BAP, then $Y$ has the BAP.
\end{corollary}

\begin{proof}
Set $X=\ell_p/Y$ and consider the exact sequence $$
\begin{CD}
0@>>> Y@>\jmath >> \ell_p@> \varpi >> X@>>> 0.
\end{CD}
$$
Now, use Proposition~\ref{prop:vscop} to put it into a diagram
$$
\xymatrixcolsep{3.5pc}\xymatrix{
0\ar[r] & \co^{(p)}X\ar[r]^-\imath \ar[d]_T & \co^{(p)}\!X\!\oplus_\mho \!X \ar[r]^-\pi \ar[d]_S & X\ar[r] \ar@{=}[d]  & 0\\
0\ar[r] & Y\ar[r]^\jmath  &\ell_p \ar[r]^\varpi & X\ar[r] & 0
}
$$
Lemma~\ref{lem:diagonal} provides an exact sequence
$$
\begin{CD}
0@>>> \co^{(p)}(X)@>>> Y\oplus\left(\co^{(p)}(X)\oplus_\mho X\right) @>>>  \ell_p@>>> 0
\end{CD}
$$
which splits, since $\ell_p$ is projective in the category of $p$-Banach spaces. In particular $ Y\oplus\left(\co^{(p)}(X)\oplus_\mho X\right)$ is isomorphic to
$\co^{(p)}(X)\oplus \ell_p$ and since $\co^{(p)}(X)$ and $\ell_p$ have the BAP so $Y$ and $\co^{(p)}(X)\oplus_\mho X$ have it.
\end{proof}

\subsection{What about the converse?}
The converse of the preceding Corollary is know to be false for $p=1$ by results of Szankowski \cite{szan}: there are subspaces $Y$ of $\ell_1$ having the BAP such that $\ell_1/Y$ fails it; cf. \cite[pp. 258--259]{CM-Isr}. For $p<1$ the situation is even comical: Kalton proved in \cite{k-loc} that the kernel of any quotient map $Q:\ell_p\lra L_p$ has a basis! Note that $L_p$ fails the BAP in a very strong way since its dual is zero.

By the way, the ideas of \cite{k-loc} can be used to generalize Corollary~\ref{cor:theend} for the so-called discrete $\mathscr L_p$-spaces for $p<1$. We refer the reader to \cite{k-loc} for all unexplained terms, in particular for the definition of the $\mathscr L_p$-spaces when $p<1$ Here, we only recall that the kernel of any quotient map $\ell_p\lra L_p$ is a discrete $\mathscr L_p$-space not isomorphic to $\ell_p$. Suppose $Q: \mathscr D_p\lra X$ is a quotient map, where $\mathscr D_p$ is a discrete $\mathscr L_p$-space, $X$ has the BAP and $p<1$. Set $Y=\ker Q$. Proceeding as in the proof of Corollary~\ref{cor:theend}, replacing $\ell_p$ by $\mathscr D_p$, we arrive to an exact sequence of $p$-Banach spaces
$$
\begin{CD}
0@>>> \co^{(p)}(X)@>>> Y\oplus\left(\co^{(p)}(X)\oplus_\mho X\right) @>>>  \mathscr D_p@>>> 0
\end{CD}
$$
This sequence splits locally (every sequence of $p$-Banach spaces whose quotient space is a $\mathscr L_p$-space does). Besides, $\co^{(p)}(X)$ and $\mathscr D_p$ have the BAP (see \cite{k-loc}) from which it is relatively easy to obtain that the middle space
$Y\oplus\left(\co^{(p)}(X)\oplus_\mho X\right)$ has the BAP and so $Y$ has it.

\section{Spaces of almost universal complemented disposition}\label{sec:aucd}

The remainder of the paper is devoted to the study the ``largest'' separable $p$-Banach space with the BAP.
Officialy, we exhibit a separable $p$-Banach which is ``complementably universal'' for the BAP in the sense that it contains a complemented copy of each separable $p$-Banach space with the BAP.

The existence of such objects, which can be seen as the $p$-Banach counterparts of Kadec/Pe\l czy\'nski/Wojtaszczyk spaces in \cite{kade, pelcuni, p-w}, was observed long time ago by Kalton in \cite{kaltuni}, a paper dealing exclusively with ``isomorphic'' properties.

Why are we bringing up these relics? Because they can be obtained and characterized (isometrically) as spaces of ``almost universal complemented disposition''.
This notion is implicit in Garbuli\'nska-W\c egrzyn's work \cite{GW} and then made explicit and studied for Banach spaces in \cite{CM-ACUD}.

Let us recall that a $p$-Banach space $X$ is of ``almost universal disposition'' if whenever $F$ is a finite dimensional $p$-normed space, $E$ is a subspace of $F$ and $u:E\lra X$ is an isometry, then for each $\e>0$ there is an $\e$-isometry from $F$ to $X$ extending $u$.
Diagramatically:
$$
\xymatrix{
E\ar[rr]^{\text{isometry}}\ar[rd]_{\text{inclusion}} & & F \ar[ld]^{\text{$\e$-isometry}}\\
& X
}
$$
This property was first considered by Gurariy in \cite{gurari}.
We refer the reader to \cite{CGK} and \cite[Chapter~3]{Av} for further information on spaces of universal disposition and proper references.

Restricting our attention to 1-complemented subspaces,
it is quite natural to consider the following property that a $p$-Banach space $X$ may or may not have:
\medskip

[$\Game$] If $F$ is a finite dimensional $p$-normed space, $E$ is a 1-complemented subspace of $F$ and $u:E\lra X$ is an isometry with 1-complemented range, then for every $\e>0$ there is an $\e$-isometry $F\lra X$ with $(1+\e)$-complemented range extending $u$.
\medskip

This is an obvious adaptation of property (E) in \cite[p. 218]{GW}. Observe that one does not require the projections to be ``compatible'' in any sense.
For tactical reasons it is better to consider the structure embedding/projection as a whole, so let us give the pertinent definitions  and introduce some special notations that shall be used in the sequel.

\subsection{Pairs}
A pair $u=\langle u^\flat, u^\sharp\rangle$ consists of two operators $u^\flat: E\lra F$ and  $u^\sharp: F\lra E$ such that $u^\sharp u^\flat={\bf 1}_E$. Thus, $u^\flat$ is an embedding of $E$ into $F$ and $u^\sharp$ is a projection along $u^\flat$.

We find most comfortable to use the notation  $
u: \xymatrix{
E  \ar@<0.3ex>[r]
& F \ar@<0.3ex>@{.>}[l] }
$
for pairs, with the understanding that the ``solid'' arrow represents the embedding part $u^\flat$, the ``dotted'' arrow is the projection part $u^\sharp$, the space $E$ is the ``domain'' of $u$ and $F$ is the ``codomain''.

The composition of $u: \xymatrix{
E  \ar@<0.3ex>[r]
& F \ar@<0.3ex>@{.>}[l] }
$ and $v: \xymatrix{
F  \ar@<0.3ex>[r]
& G \ar@<0.3ex>@{.>}[l] }
$
is, as one can expect, $v\circ u= \langle v^\flat  u^\flat, u^\sharp  v^\sharp\rangle$.

We measure the ``size'' of a pair taking $\|u\|=\max\left( \|u^\flat\|, \|u^\sharp\|\right)$.
Note that $\|u\|\geq 1$ (unless $E=0$) and that $\|u\|\leq 1+\e$ implies that $u^\flat$ is an $\e$-isometry. If $\|u\|=1$ we say that $u$ is contractive.

\begin{definition}\label{def:aucd}
A $p$-normed space $X$ is said to be of almost universal complemented disposition if for all contractive pairs $
u: \xymatrix{
E  \ar@<0.3ex>[r]
& X \ar@<0.3ex>@{.>}[l] }
$
and $
v: \xymatrix{
E  \ar@<0.3ex>[r]
& F \ar@<0.3ex>@{.>}[l] }
$, where $F$ is a finite dimensional $p$-normed space, and every $\e>0$, there exist a pair $
w: \xymatrix{
F  \ar@<0.3ex>[r]
& X \ar@<0.3ex>@{.>}[l] }
$ such that $u=w\circ v$ and $\|w\|\leq 1+\e$.
\end{definition}

The situation is illustrated by the following diagram where both the solid arrows (emdeddings) and the dotted arrows (projections) commute:
$$
\xymatrixcolsep{3.5pc}\xymatrix{
E  \ar@<0.3ex>[rr]^{v^\flat} \ar@<0.3ex>[dr]^{u^\flat}
&  & F \ar@<0.3ex>@{.>}[ll]^{v^\sharp} \ar@<0.3ex>[ld]^{w^\flat}  \\
& X \ar@<0.3ex>@{.>}[ur]^{\quad w^\sharp}  \ar@<0.3ex>@{.>}[ul]^{u^\sharp}
 }
$$
Hence the notion of a space of almost universal complemented disposition is formally stronger than $[\Game]$.

Note that, according to our definitions, the ``null pair''
$
\xymatrix{
0  \ar@<0.3ex>[r]
& F \ar@<0.3ex>@{.>}[l] }
$
is contractive. This excludes from Definition~\ref{def:aucd} spaces with trivial dual  and prevents them from having property $[\Game]$.

\subsection{Two correction lemmas}
This Section contains a couple of results that allow us to gain some flexibility in the main proofs. The first one is not really concerned with pairs although every linear surjective isomorphism $f:X\lra Y$ can be ``expanded'' to a pair $
\langle f,f^{-1}\rangle: \xymatrix{
X  \ar@<0.3ex>[r]
& Y \ar@<0.3ex>@{.>}[l] }
$.

\begin{lemma}[Small automorphisms]\label{lem:small}
Let $E$ be a finite dimensional subspace of a $p$-Banach space $X$ which is complemented throught a projection $P$ and let $e_1,\ldots, e_k$ be a normalized basis of $E$.

For every $\e>0$ there is $\delta>0$, depending on $\e$ and $E$, such that if $x_i\in X$ are such that $\|e_i-x_i\|<\delta$ for $1\leq i\leq k$ and  we define a linear map $f:X\lra X$ by
$$
f(x)=\begin{cases}
x_i & \text{ if $x=e_i$ for $1\leq i\leq k$}\\
x &\text{ if $x\in\ker P$}
\end{cases}
$$
then $\|f-{\bf 1}_X\|<\e$.
\end{lemma}

\begin{proof}Take $K$ so large that $\left(\sum_i|\lambda_i|^p\right)^{1/p}\leq K\left\|\sum_i\lambda_ie_i\right\|$. Pick $x\in X$ and write $x=y+z$, with $y=Px$ and then $y=\sum_i\lambda_ie_i$. Then since $z\in\ker P$ one has
$$
\|fx-x\|=  \|fy-y\|= \left\|\sum_i\lambda_ix_i- \sum_i\lambda_ie_i\right\|
\leq \delta \left(\sum_i|\lambda_i|^p\right)^{1/p}\!\!\leq \delta K \|y\|\leq \delta K \|P\|\|x\|.
$$
Hence $\delta=\e/(K\|P\|)$ suffices.
\end{proof}

The hypothesis on $E$ is necessary: there are $p$-Banach spaces whose only endomorphisms are the scalar multiples of the identity \cite{rigid}.

\begin{lemma}[Correction of the bound]\label{lem:bound}
If $
u: \xymatrix{
E  \ar@<0.3ex>[r]
& F \ar@<0.3ex>@{.>}[l] }
$ is a pair with $\|u\|\leq 1+\e$, then there is a $p$-norm $|\cdot|$ on $F$ such that
\begin{equation}\label{ineq}
(1+\e)^{-1}\|f\|\leq |f| \leq (1+\e) \|f\|\quad\quad(f\in F)
\end{equation}
and $u$ becomes contractive when the original $p$-norm of $F$ is replaced by $|\cdot|$.
\end{lemma}

\begin{proof}
The hypotheses imply that $u^\flat$ is an $\e$-isometry. The unit ball of the new $p$-norm of $F$ has to be the $p$-convex hull of the set
$$
u^\flat[B_E]\bigcup (1+\e)^{-1}B_F.
$$
Hence we define
$$
|f|= \inf\left\{ \left(	\|x\|^p+ (1+\e)^p\|g\|^p\right)^{1/p}: f=u^\flat(x)+g, x\in E, g\in F \right\}.
$$
Clearly, this is $p$-norm. Let us see that everything works fine. First, taking $x=0$ and $g=f$ we have $|f|\leq  (1+\e) \|f\|$. The other inequality of (\ref{ineq}) is as follows: if $f=u^\flat(x)+g$, then
$$
\|x\|^p+(1+\e)^p\|g\|^p=  \|x\|^p+(1+\e)^p\|f-u^\flat(x)\|^p\geq \frac{\|u^\flat x\|^p +\|f-u^\flat(x)\|^p}{(1+\e)^p}\geq \frac{\|f\|^p}{(1+\e)^p},
$$
hence $|f|\geq (1+\e)^{-1}\|f\|$.

Let us compute the ``new'' quasinorms of $u^\flat$ and $u^\sharp$.

Given $x\in E$ one has $|u^\flat x|^p\leq \|x\|^p$, so the quasinorm of $u^\flat$  is at most 1. Actually it is clear that  $|u^\flat x|= \|x\|$ for all $x\in E$. Indeed we have
$$
\begin{aligned}
|u^\flat x|^p&= \inf\left\{	\|y\|^p+ (1+\e)^p\|g\|^p: u^\flat x=u^\flat(y)+g, y\in E, g\in F \right\}\\
&=  \inf\left\{	\|y\|^p+ (1+\e)^p\|u^\flat(x-y)\|^p: y\in E \right\}\\
&\geq   \inf\left\{	\|y\|^p+ \|x-y\|^p: y\in E \right\}\\
& =\|x\|^p.
\end{aligned}
$$
Finally we check that
$$
|u^\sharp|=\sup_{|f|\leq 1}\|u^\sharp f\|= \sup_{|f|< 1}\|u^\sharp f\| \leq  1.
$$
If $|f|<1$, we can write $f=u^\flat(x)+ g$, with $\|x\|^p+(1+\e)^p\|g\|^p<1$. Hence
$$
\begin{aligned}
\|u^\sharp f\|&= \|u^\sharp(u^\flat x+ g)\|
= \|x+u^\sharp g\|\\
&\leq \left(\|x\|^p+\|u^\sharp g\|^p\right)^{1/p}\leq \left(\|x\|^p+(1+\e)^p\|g\|^p\right)^{1/p}< 1.
\end{aligned}\vspace{-22pt}
$$
\end{proof}

\subsection{The category of allowed pairs}\label{sec:allowed}

We now define the restricted class of pairs between finite dimensional spaces that will be used in the main construction.

A point $x\in \mathbb K^n$ is said to be rational if all its coordinates are rational. When $\mathbb K=\mathbb C$ this means that both the real and imaginary parts are rational numbers.
A linear map $f:\K^n\lra \K^m$ is said to be rational if it carries rational points  into rational points.

 A rational $p$-norm on $\mathbb K^n$ is one whose unit ball is the $p$-convex hull of a finite set of rational points. Thus, a rational $p$-norm is given by the formula
$$
|x|=\inf\left\{\left(\sum_i|\lambda_i|^p\right)^{1/p}: x =\sum_i\lambda_i x_i\right\}
$$
where $(x_i)$ is a finite set of rational points.

For each $n\in\N$ let $\mathcal{N}_n$ be the set of all $p$-norms on $\K^n$, where $\K^0$ is understood as $0$. Put
$$\mathcal N= \bigsqcup_{n\geq 0} \mathcal{N}_n.$$

We define a class of $p$-norms, that for lack of a better name we call ``allowed $p$-norms'' (formally, a subset of $\mathcal{N}$) recursively, as follows:

\begin{enumerate}
\item Each rational $p$-norm is allowed.
\item If $f:\K^n\lra \K^m$ is rational and injective and $|\cdot|$ is an allowed $p$-norm on $\K^m$, then $\|x\|=|f(x)|$ is an allowed $p$-norm on $\K^n$.
\item If $|\cdot|_1$ and $|\cdot|_2$ are allowed $p$-norms on $\K^n$ and $\K^m$, respectively, then the direct product $\|(x,y)\|=\max(|x|_1, |y|_2)$ is an allowed $p$-norm on $\K^{n+m}$.
\item If $|\cdot|_1$ and $|\cdot|_2$ are allowed $p$-norms on $\K^n$ and $\K^m$ respectively and $f:\K^n\lra \K^m$ is a rational map, then, for every rational number $\e>0$, the $p$-norm
$$
\|y\|= \inf\left\{ \left(	|x|_1^p+ (1+\e)^p|z|_2^p\right)^{1/p}: y=f(x)+z, x\in \K^n, g\in \K^m \right\}.
$$
is allowed on $\K^m$.
\end{enumerate}

These conditions are just the minimal set of requirements we need to make work some forthcoming tricks.
An allowed space is just the direct product of finitely many copies of the ground field furnished with an allowed $p$-norm.

Finally, declare a contractive  pair $
u: \xymatrix{
E  \ar@<0.3ex>[r]
& F \ar@<0.3ex>@{.>}[l] }
$
allowed if $E$ and $F$ are allowed $p$-normed spaces and both $u^\flat$ and $u^\sharp$ are rational maps.

Clearly, the allowed pairs form a countable category.

\subsection{Amalgamating pairs}
We now stablish that pairs have the so-called ``amalgamation property''. This just means that each diagram of pairs
$$
\xymatrix{
E  \ar@<0.3ex>[rr] \ar@<0.3ex>[d]
&  & F \ar@<0.3ex>@{.>}[ll]  \\
G  \ar@<0.3ex>@{.>}[u]
 }
$$
can be ``completed '' to a commutative diagram of pairs
$$
\xymatrix{
E \ar@<0.3ex>[rr] \ar@<0.3ex>[d]
&  & F \ar@<0.3ex>@{.>}[ll]  \ar@<0.3ex>[d] \\
G  \ar@<.3ex>@{.>}[u] \ar@<0.3ex>[rr]
&  &H \ar@<0.3ex>@{.>}[ll] \ar@<0.3ex>@{.>}[u] }
$$
This can be obtained in several ways. For our present purposes the most convenient one is to use the pull-back construction in the setting of $p$-Banach spaces, as described in the Appendix.

\begin{lemma}[Pull-back with pairs]\label{lem:PBwithpairs} Given pairs $
u: \xymatrix{
E  \ar@<0.3ex>[r]
& F \ar@<0.3ex>@{.>}[l] }
$
 and  $
v: \xymatrix{
E  \ar@<0.3ex>[r]
& G \ar@<0.3ex>@{.>}[l] }
$
there are pairs $\overline{u}=\langle \overline{u}^\flat, \overline{u}^\sharp\rangle$ and $\overline{v}=\langle \overline{v}^\flat, \overline{v}^\sharp\rangle$ such that the following diagram commutes
$$
\xymatrix{
E  \ar@<0.3ex>[rr]^{u^\flat} \ar@<0.3ex>[d]^{v^\flat}
&  & F \ar@<0.3ex>@{.>}[ll]^{u^\sharp}  \ar@<0.3ex>[d]^{\overline{v}^\flat} \\
G  \ar@<0.3ex>@{.>}[u]^{v^\sharp} \ar@<0.3ex>[rr]^{\overline{u}^\flat}
&  & H \ar@<0.3ex>@{.>}[ll]^{\overline{u}^\sharp} \ar@<0.3ex>@{.>}[u]^{\overline{v}^\sharp} }
$$
Moreover:
\begin{itemize}
\item If $u$ and $v$ are contractive, then so are $\overline{u}$ and $\overline{v}$.
\item If $u$ is contractive and $\|v^\flat\|\leq 1$, then $\overline{u}$ is contractive and $\|\overline{v}^\sharp\|\leq\| {v}^\sharp\| $.
\item If $u$ and $v$ are allowed pairs, then so are $\overline{u}$ and $\overline{v}$.
\end{itemize}
\end{lemma}

\begin{proof}
The proof is based on the properties  the pull-back construction, as presented in the Appendix.

We start with the projections $u^\sharp $ and $v^\sharp$ and we set $H=\PB$, their pull-back space. In this way we obtain the commutative diagram
$$
\xymatrix{
E
&  & F \ar@<0.1ex>@{.>}[ll]^{u^\sharp}   \\
G  \ar@<0.1ex>@{.>}[u]^{v^\sharp}
&  & \PB \ar@<0.1ex>@{.>}[ll]^{\overline{u}^\sharp} \ar@<0.1ex>@{.>}[u]^{\overline{v}^\sharp} }
$$
The embedding $\overline{u}^\flat$ is provided by the universal property of the pull-back construction applied to the couple ${\bf 1}_G, u^\flat v^\sharp$:
$$
\xymatrix{
E
&  & F \ar@<0.1ex>@{.>}[ll]^{u^\sharp}   \\
G  \ar@<0.1ex>@{.>}[u]^{v^\sharp}
&  & \PB \ar@<0.1ex>@{.>}[ll]_{\overline{u}^\sharp} \ar@<0.3ex>@{.>}[u]^{\overline{v}^\sharp} \\
 & & & &G \ar@<0.0ex>@{.>}[ull]_{\overline{u}^\flat}  \ar@{.>}@/_/[uull]_{u^\flat v^\sharp}
 \ar@{:}@/^/[ullll]^{{\bf 1}_G}
}
$$
while  $\overline{v}^\flat$ is obtained from the couple ${\bf 1}_F, v^\flat u^\sharp$. We have
\begin{itemize}
\item $\| \overline{u}^\flat\|\leq \|u^\flat\|\| v^\sharp\|$
\item $\| \overline{v}^\flat\|\leq \|v^\flat\|\| u^\sharp\|$
\item $ \overline{u}^\sharp  \overline{u}^\flat={\bf 1}_G$, that is, $\overline{u}=\langle \overline{u}^\flat, \overline{u}^\sharp\rangle$ is a pair.
\item $ \overline{v}^\sharp  \overline{u}^\flat= {u}^\flat v^\sharp$
\item $ \overline{v}^\sharp  \overline{v}^\flat={\bf 1}_F$, that is, $\overline{v}=\langle \overline{v}^\flat, \overline{v}^\sharp\rangle$ is a pair.
\item $ \overline{u}^\sharp  \overline{v}^\flat= {v}^\flat  u^\sharp$
\end{itemize}
It only remains to see that the embeddings commute, that is: $\overline{v}^\flat  u^\flat= \overline{u}^\flat  v^\flat$. This can be deduced from the uniqueness part of the universal property of the pull-back construction. Indeed, regarding Diagram~\ref{diag:PB}, we have that since $u^\sharp  u^\flat= v^\sharp v^\flat$ (they are the identity on $E$) there must be a unique operator $\gamma:E\lra \PB$ making commutative the diagram
$$
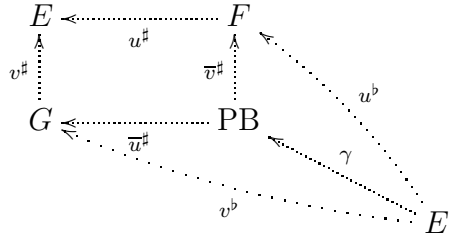
\xymatrix{
E
&  & F \ar@<0.1ex>@{.>}[ll]^{u^\sharp}   \\
G  \ar@<0.1ex>@{.>}[u]^{v^\sharp}
&  & \PB \ar@<0.1ex>@{.>}[ll]^{\overline{u}^\sharp} \ar@<0.3ex>@{.>}[u]^{\overline{v}^\sharp} \\
 & & & &E \ar@<0.0ex>@{.>}[ull]_{\gamma}  \ar@{.>}@/_/[uull]_{u^\flat}
 \ar@{.>}@/^/[ullll]^{v^\flat}
}
$$
But since both $\overline{v}^\flat  u^\flat$ and $\overline{u}^\flat  v^\flat$
do the trick we conclude that they agree.

This also proves the first two ``moreover'' statements. The third one follows from the final remark in the Appendix, after representing $\PB$ as an allowed space.
\end{proof}

\subsection{Pairs with allowed domain}

\begin{lemma}[Density of allowed pairs]\label{lem:density}
Given a contractive pair $
u: \xymatrix{
U  \ar@<0.3ex>[r]
& F \ar@<0.3ex>@{.>}[l] }
$, with allowed domain $U$ and $\varepsilon>0$ there is an allowed pair $
u_0: \xymatrix{
U  \ar@<0.3ex>[r]
& F_0 \ar@<0.3ex>@{.>}[l] }
$
and a $\varepsilon$-isometry $g:F\lra F_0$ making commutative the diagram
$$
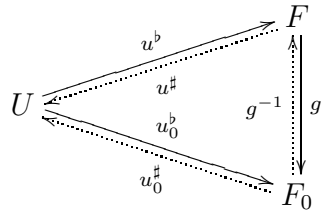
\xymatrixcolsep{2pc}\xymatrixrowsep{1.5pc}\xymatrix{& & &  F  \ar@<0.3ex>@{.>}[llld]^{u^\sharp} \ar@<0.3ex>[dd]^{g} \\
U  \ar@<0.3ex>[rrru]^{u^\flat} \ar@<0.3ex>[rrrd]^{u_0^\flat}  \\
&&  & F_0  \ar@<0.3ex>@{.>}[lllu]^{u_0^\sharp} \ar@<0.3ex>@{.>}[uu]^{g^{-1}}
 }
$$
\end{lemma}

\begin{proof}
We may assume that $\e$ is rational.
Let $(e_i)_{1\leq i\leq n}$ be the unit basis of $U=\K^n$ and let $(f_j)_{1\leq j\leq m}$ be a basis of $\ker u^\sharp$. Then the set $\{u^\flat(e_1),\dots, u^\flat(e_n), f_1, \dots, f_m\}$ is a basis of $F$ and can be used to define an isomorphism $g:F\lra \K^{n+m}$. Take a rational $p$-norm $|\cdot|_0$ on $\K^{n+m}$ making $g$ an $\e$-isometry:
$$
(1+\e)^{-1}\|y\|	\leq |g(y)|_0\leq (1+\e)\|y\|\quad\quad(y\in F).
$$
Consider the pair $u_0=\langle g,g^{-1}\rangle\circ u$. Then $u_0$ is rational (actually: $u^\flat(x)=(x,0)$ and $u^\sharp(x,y)=x$) and
$$
\left\| u_0: \xymatrix{
U  \ar@<0.3ex>[r]
& (\K^{n+m}, |\cdot|_0) \ar@<0.3ex>@{.>}[l] }\right\|\leq 1+\e.
$$
Finally, we define a new $p$-norm on $\K^{n+m}$ by the formula
$$
|y|= \inf\left\{ \big{(}	\|x\|^p+ (1+\e)^p|z|_0^p \big{)}^{1/p}: y=u_0^\flat(x)+z, x\in \K^n, z\in \K^{n+m} \right\}.
$$
This $p$-norm has to be allowed on $\K^{n+m}$ (by the the fourth condition of the list),
satisfies the estimate  $$
(1+\e)^{-1} |y|_0	\leq |y|\leq (1+\e)|y|_0\quad\quad(y\in \K^{n+m})
$$
 and makes $u_0$ into a contractive  pair (see the proof of Lemma~\ref{lem:bound}) which is therefore allowed. Hence, if $F_0$ is  $\K^{n+m}$ equipped with $|\cdot|$ we have
$$
(1+\e)^{-2}\|y\|	\leq |g(y)|\leq (1+\e)^2\|y\|\quad\quad(y\in F),
$$
which ends the proof.
\end{proof}

\subsection{A Fra\"\i ss\'e sequence of allowed pairs.} The category of allowed pairs is countable and, by Lemma~\ref{lem:PBwithpairs}, it admits amalgamations. It follows from general results that  it  has a Fra\"\i ss\'e sequence. We are going to construct such an object ``by hand'' as follows.

\begin{proposition} There is a sequence of allowed pairs $u_n: \xymatrix{
U_n  \ar@<0.3ex>[r]
& U_{n+1} \ar@<0.3ex>@{.>}[l] }
$
having the following property:
For every allowed pair $
v: \xymatrix{
U_n  \ar@<0.3ex>[r]
& F \ar@<0.3ex>@{.>}[l] }
$ there is $m>n$ and an allowed pair  $
u: \xymatrix{
F  \ar@<0.3ex>[r]
& U_m \ar@<0.3ex>@{.>}[l] }
$
such that $u\circ v$ is the bonding pair $
 \xymatrix{
U_n  \ar@<0.3ex>[r]
& U_m \ar@<0.3ex>@{.>}[l] }
$.
\end{proposition}
The commutative diagram of pairs
$$
\xymatrix{
U_0  \ar@<0.3ex>[r]
& U_1 \ar@<0.3ex>[r] \ar@<0.3ex>@{.>}[l] &... \ar@<0.3ex>[r] \ar@<0.3ex>@{.>}[l]
& U_n \ar@<0.3ex>@{.>}[l] \ar@<0.3ex>[r] \ar@<0.3ex>[rd] \ar@<0.3ex>@{.>}[l] &... \ar@<0.3ex>[r] \ar@<0.3ex>@{.>}[l]
& U_m \ar@<0.3ex>@{.>}[l] \ar@<0.3ex>@{.>}[ld] \ar@<0.3ex>[r] \ar@<0.3ex>@{.>}[l] &...  \ar@<0.3ex>@{.>}[l] \\
& & & & F \ar@<0.3ex>@{.>}[ul]  \ar@<0.3ex>[ur]
}
$$
illustrates the relevant property of the sequence we want to construct.


\begin{proof}
As there are only countable many allowed pairs, we can take a sequence $(f_n,k_n)$ passing through all couples of the form $(f,k)$, where $f$ is an  allowed pair and $k\in\mathbb N$, in such a way that each $(f,k)$ appears infinitely many times. For instance, if $(g_n)$ is an enumeration of the allowed pairs we can take
$$\begin{array}{l}
(g_1,1);\\
(g_1,1); (g_1,2); (g_2,1); \\
(g_1,1); (g_1,2); (g_2,1); (g_1, 3); (g_2,2); (g_3,1);\\
\dots
\end{array}
$$
The sequence $(u_n)$ is constructed by induction. The initial pair is the obvious one $
u_0: \xymatrix{0  \ar@<0.3ex>[r]
& \mathbb K \ar@<0.3ex>@{.>}[l] }
$. Having defined $
u_{n-1}: \xymatrix{
U_{n-1}  \ar@<0.3ex>[r]
& U_n \ar@<0.3ex>@{.>}[l] }
$
we take a look at $(f_n, k_n)$, which consists of a ``number of control'' $k_n$ and a pair $
f_n: \xymatrix{
E_n  \ar@<0.3ex>[r]
& F_n \ar@<0.3ex>@{.>}[l] }
$.
If either $k_n\geq n$ or the ``domain'' of $f_n$ (the space $E_n$) is not $U_{k_n}$, then we set $U_n=U_{n-1}$ and $u_n$ is the identity of $U_{n-1}$.

Otherwise we have $k_n<n$ and the domain  $f_n$  is  $U_{k_n}$. Thus we have two pairs with domain $E_n=U_{k_n}$, namely the bonding morphism $\kappa : \xymatrix{
U_{k_n}  \ar@<0.3ex>[r]
& U_n \ar@<0.3ex>@{.>}[l] }
$
and $f_n$ itself. Thanks to Lemma~\ref{lem:PBwithpairs} these fit into a commutative diagram of allowed pairs
$$
\xymatrix{
U_{k_n}=E_n  \ar@<0.3ex>[rr]^{f_n^\flat} \ar@<0.3ex>[d]
&  & F_n \ar@<0.3ex>@{.>}[ll]  \ar@<0.3ex>[d]^{\overline{\kappa}} \\
U_n  \ar@<0.3ex>@{.>}[u]^{\kappa} \ar@<0.3ex>[rr]^{\overline{f}_n}
&  & \PB \ar@<0.3ex>@{.>}[ll] \ar@<0.3ex>@{.>}[u] }
$$
Then set $U_{n+1}=\PB$ and $u_n=\overline{f}_n$. This completes the induction step.

Let us check that the resulting sequence $(u_n)_{n\geq 0}$ has the required property.
Fix $n\in \mathbb N$ and let $v: \xymatrix{
U_{n}  \ar@<0.3ex>[r]
& F \ar@<0.3ex>@{.>}[l] }
$ be an allowed pair. Take $m>n$ such that $(f_m,k_m)=(v,n)$. Then the $(m-1)$-th pair of the sequence $(u_n)_{n\geq 0}$ arises from the pull-back diagram
$$
\xymatrix{
U_{k_m}=U_n=E_m  \ar@<0.3ex>[rr]^{v} \ar@<0.3ex>[d]
&  & F \ar@<0.3ex>@{.>}[ll]  \ar@<0.3ex>[d]^{\overline{\kappa}} \\
U_{m-1}  \ar@<0.3ex>@{.>}[u]^{\kappa} \ar@<0.3ex>[rr]^{\overline{v}=u_{m-1}}
&  & \PB=U_m \ar@<0.3ex>@{.>}[ll] \ar@<0.3ex>@{.>}[u] }
$$
and so $\overline{\kappa}\circ v$ is the bonding pair $\xymatrix{
U_{n}  \ar@<0.3ex>[r]
& U_m \ar@<0.3ex>@{.>}[l] }
$.
\end{proof}

\subsection{A space of almost universal complemented disposition}
Let $\mathscr{K}_p$ be the direct limit of the inductive system
underlying $(u_n)$:
$$
\begin{CD}
0 @>>> \mathbb K @>{u_1^\flat}>>  U_2 @>>> \cdots  @>>> U_{n} @>{u_n^\flat}>> U_{n+1}@>>>\cdots
\end{CD}
$$
\begin{theorem}
The space $\mathscr{K}_p$ is of almost universal complemented disposition.
\end{theorem}

\begin{proof}
Let $u: \xymatrix{
E \ar@<0.3ex>[r]
& \mathscr K_p \ar@<0.3ex>@{.>}[l] }
$ and $v: \xymatrix{
E \ar@<0.3ex>[r]
& F \ar@<0.3ex>@{.>}[l] }
$
be contractive pairs, where $F$ is a finite dimensional $p$-normed space
and $0<\e<1$.

Our immediate aim is to push $u$ into some $U_n$ at the cost of spoiling the isometric character of the embedding and the bound of the projection.

To this end note that since the union of the subspaces $U_n$ is dense in $\mathscr K_p$ a straighforward application of Lemma~\ref{lem:small} provides an integer $n$ and an automorphism $f$ of $\mathscr K_p$ such that $f[u^\flat[E]]\subset U_n$ with $\|f-{\bf 1}_{\mathscr K_p}\|<\e$ and $\max\left(\|f\|,\|f^{-1}\|\right)<1+\e$.

After mulpliplying $f$ by $\|f^{-1}\|$ and dividing $f^{-1}$ by $\|f^{-1}\|$ we may assume and do that $\|f^{-1}\|=1$, with $\|f\|<(1+\e)^2.$

Then $\langle f,f^{-1}\rangle\circ u$ is a pair from $E$ to $\mathscr K_p$ that
``factors'' through the natural pair $\imath_n: \xymatrix{
U_n \ar@<0.3ex>[r]
& \mathscr K_p \ar@<0.3ex>@{.>}[l] }
$
in the sense that  $\langle f,f^{-1}\rangle\circ u=\imath_n\circ u_*$, where
 $u_*: \xymatrix{
E \ar@<0.3ex>[r]
& U_n \ar@<0.3ex>@{.>}[l] }
$
is defined as
$$
u_*^\flat=\imath_n^\sharp f  u^\flat, \quad\quad u_*^\sharp=u^\sharp f^{-1} \imath_n^\flat.
$$
Indeed $u_*^\sharp$ is a projection along $u_*^\flat$ since
$
u_*^\sharp u_*^\flat= u^\sharp  f^{-1} \imath_n^\flat  \imath_n^\sharp f u^\flat= {\bf 1}_E
$.
Now we work with this $u_*$ and return to $u$ at the end of the proof.

Let us amalgamate $u_*$ and $v$ in the pull-back diagram
$$
\xymatrix{
E  \ar@<0.3ex>[rr]^{v^\flat} \ar@<0.3ex>[d]^{u_*^\flat}
&  & F \ar@<0.3ex>@{.>}[ll]^{v^\sharp}  \ar@<0.3ex>[d]^{\overline{{u}}_*^\flat} \\
U_n  \ar@<0.3ex>@{.>}[u]^{u_*^\sharp} \ar@<0.3ex>[rr]^{\overline{v}^\flat}
&  & \PB \ar@<0.3ex>@{.>}[ll]^{\overline{v}^\sharp} \ar@<0.3ex>@{.>}[u]^{\overline{u}_*^\sharp}
}
$$
Note that since $\|u_*^\sharp\|\leq 1$ the lower pair $\overline{v}=\langle \overline{v}^\flat,  \overline{v}^\sharp\rangle $ is contractive.
Then we apply Lemma~\ref{lem:density} to $\overline{v}$ to obtain an allowed space $F_0$   together with a  $\e$-isometry $g:\PB\lra F_0$ such that $\overline{v}_0= \langle g,g^{-1}\rangle\circ\overline{v}$ is an allowed pair. Finally, by the very definition of the sequence $(u_n)_{n\geq 0}$ there is $m>n$ and an allowed pair  $w_0 : \xymatrix{
F_0 \ar@<0.3ex>[r]
& U_m \ar@<0.3ex>@{.>}[l] }
$ such that $w_0\circ\overline{v}_0$ is the bonding pair $ \xymatrix{
U_n \ar@<0.3ex>[r]
& U_m \ar@<0.3ex>@{.>}[l] }
$, so we have the following commutative diagram of pairs:
$$
\xymatrix{
E  \ar@<0.3ex>[rr]^{v^\flat} \ar@<0.3ex>[d]^{{u}_*^\flat}
&  & F \ar@<0.3ex>@{.>}[ll]^{v^\sharp}  \ar@<0.3ex>[d]^{\overline{{u}}_*^\flat} \\
U_n \ar@<0.3ex>[d] \ar@<0.3ex>@{.>}[u] \ar@<0.3ex>[rr]^{\overline{v}^\flat}
&  & \PB \ar@<0.3ex>@{.>}[ll]^{\overline{v}^\sharp} \ar@<0.3ex>@{.>}[u]  \ar@<0.3ex>[d]^{g\quad\quad} \\
U_m \ar@<0.3ex>@{.>}[u]^{{}^\text{bonding }_\text{\quad pair}} \ar@<-0.3ex>@{.>}[rr]_{w_0^\sharp}  & & F_0 \ar@<0.3ex>@{.>}[u]^{g^{-1}} \ar@<-0.3ex>[ll]_{{w}_0^\flat}
}
$$
In particular one has
$$
w_0\circ \langle g,g^{-1}\rangle \circ\overline{u}\circ v= u_*= \langle f,f^{-1}\rangle\circ u,
$$
and letting $w= \langle f^{-1},f\rangle\circ w_0\circ \langle g,g^{-1}\rangle \circ\overline{u}_*$ we are done since $w\circ v= u$ and
$$
\|w\|\leq  \|\langle f^{-1},f\rangle\|\| w_0\|\| \langle g,g^{-1}\rangle \|\|\overline{u}_*\|\leq (1+\e)^3< 1+4\e.
\vspace{-16pt}$$
\end{proof}

\section{Further properties of $\mathscr{K}_p$}

In this Section we study a number of isometric properties of $\mathscr{K}_p$. In general these upgrade well-known isomorphic properties of Kadec/Pe\l czy\'nski/Wojtaszczyk and Kalton spaces.

The key fact is the following result that allows us to recover  ``approximate pairs'' (couples of operators  $f^\dagger: X\lra Y$ and $f^\ddagger: Y\lra X$ whose composition is close to the identity of $X$) as composition of the arrows of two pairs with a common, \emph{ad hoc} codomain.

\begin{lemma}\label{lem:W}
Let  $f^\dagger: X\lra U$ and $f^\ddagger: U\lra X$ be contractive operators such that $\|f^\ddagger f^\dagger-{\bf 1}_X\|\leq \e$.  Then there is a $p$-Banach space $W$ and contractive  pairs
 $\alpha: \xymatrix{
X  \ar@<0.3ex>[r]
& W \ar@<0.3ex>@{.>}[l] }
$
and
 $\beta: \xymatrix{
U \ar@<0.3ex>[r]
& W \ar@<0.3ex>@{.>}[l] }
$
such that $f^\dagger= \beta^\sharp \alpha^\flat, f^\ddagger= \alpha^\sharp \beta^\flat$ and $\|\alpha^\flat-\beta^\flat f^\dagger\|\leq\e$.
\end{lemma}

The relevant diagram is
$$
\xymatrix{
X  \ar@<0.4ex>[rrrd]^{\alpha^\flat}   \ar@<0.4ex>[dd]^{f^\dagger} \\
&&& W  \ar@<0.3ex>@{.>}[ulll]^{\alpha^\sharp}   \ar@<0.3ex>@{.>}[dlll]^{\beta^\sharp} \\
U  \ar@<0.4ex>[rrru]^{\beta^\flat}   \ar@<0.4ex>[uu]^{f^\ddagger}
}
$$
This result is just the $p$-normed version of  \cite[Lemma 3.2]{GW} and depends on \cite[Lemma 3.1]{CGK} in the same way as
 \cite[Lemma 3.2]{GW} depends on \cite[Lemma 2.1]{KS}, which is the most common recent ancestor of all of them.

\begin{proof}
We present a lightweight proof, based on \cite[Lemma 3.1]{CGK}.
The space $W$ is just the direct sum $X\oplus U$ equipped with the $p$-norm
$$
\|(x,y)\|=\inf\left\{\left(\|x_0\|^p+\|y_1\|^p+\e^p\|x_2\|^p\right)^{1/p}: (x,y)=(x_0,0)+(0,y_1)+(x_2,-f^\dagger x_2)		\right\}.
$$
It is \emph{really} easy to see that $\|(x,0)\|=\|x\|$ and $\|(0,y)\|=\|y\|$ for every $x\in X$ and every $y\in U$. Thus, letting $\alpha^\flat(x)=(x,0)$ and
$\beta^\flat(y)=(0,y)$ we quickly obtain that $\|\alpha^\flat-\beta^\flat f^\dagger\|\leq\e$.
As for the projections, we are forced to define $\alpha^\sharp(x,y)=x+f^\ddagger(y)$ and
$\beta^\sharp(x,y)=y+f^\dagger(x)$.
It is then clear that
$$
\alpha^\sharp\alpha^\flat={\bf 1}_X,\quad
\beta^\sharp\beta^\flat={\bf 1}_U,\quad
f^\dagger= \beta^\sharp \alpha^\flat,\quad f^\ddagger= \alpha^\sharp \beta^\flat.
$$
To see that $\alpha^\sharp$ and $\beta^\sharp$ are contractive, pick $(x,y)\in W$ and suppose
$$
(x,y)=(x_0,0)+(0,y_1)+(x_2,-f^\dagger(x_2)).
$$
We then have
$$
\begin{aligned}
\|\alpha^\sharp (x,y)\|&=\|x_0+ x_2+ f^\ddagger(y_1) -f^\ddagger f^\dagger(x_2)\|
\leq \left(\|x_0\|^p+\|y_1\|^p+\e^p\|x_2\|^p\right)^{1/p}\\
\|\beta^\sharp (x,y)\|&=\|f^\dagger(x_0)+y_1\|\leq  \left(\|x_0\|^p+\|y_1\|^p\right)^{1/p}\leq \left(\|x_0\|^p+\|y_1\|^p+\e^p\|x_2\|^p\right)^{1/p}
\end{aligned}
$$
and since $\|(x,y)\|$ is the infimum of the numbers that can appear in the right hand side we have $\|\alpha^\sharp\|, \|\beta^\sharp\|\leq 1$.
\end{proof}

We also need a technique to ``paste'' operators defined on a chain of subspaces. Let $A$ and $B$ be $p$-Banach spaces and $(A_n)$ a chain of subspaces whose union is dense in $A$.
Let $a_n: A_n\lra B$ be a sequence of contractive operators such that $\|a_{n+1}|_{A_n}-a_n\|\leq \e_n$, where $\sum_n\e_n^p<\infty$. Then for each $x\in A_k$ the sequence $(a_n(x))_{n\geq k}$ converges in $B$ (it is a Cauchy sequence and $B$ is complete) and so there is a unique contractive operator $a:A\lra B$ such that
$$a(x)=\lim_{n\geq k}a_n(x)\quad\quad(x\in A_k).$$
This operator shall be referred to as the ``pointwise limit'' of the sequence $(a_n)$.

\subsection{``Universality''}
A monotone finite dimensional decomposition of a $p$-Banach space $X$ (1FDD, for short) is a chain $(X_n)$ of finite dimensional subspaces of $X$ whose union is dense in $X$ and such that each $X_n$ is 1-complemented in $X_{n+1}$. These inclusions and projections can be arranged into a sequence of contractive pairs
$$
\xymatrix{
X_0  \ar@<0.3ex>[r]
& X_1 \ar@<0.3ex>[r] \ar@<0.3ex>@{.>}[l] & X_2 \ar@<0.3ex>[r] \ar@<0.3ex>@{.>}[l] & ...\ar@<0.3ex>@{.>}[l]
&
}
$$
It is clear that every space with a 1FDD has the BAP with constant 1. One has.

\begin{proposition}
Every $p$-Banach space with a 1FDD is isometric to a 1-comple\-mented subspace of $\mathscr K_p$.
\end{proposition}

\begin{proof}
Suppose $(X_n)$ is a 1FDD of $X$. For each integer $n$ we denote by $\xi_n: \xymatrix{
X_n  \ar@<0.3ex>[r]
& X_{n+1} \ar@<0.3ex>@{.>}[l] }
$
the  ``bonding'' pair, that is, $\xi^\flat_n$ is the inclusion of $X_n$ into $X_{n+1}$ and $\xi^\sharp_n: X_{n+1}\lra X_n$ is a fixed contractive projection.

Keeping the notations of Section~\ref{sec:aucd} and considering the spaces $U_n$ as subspaces of $\mathscr K_p$, we shall construct an increasing sequence of integers $(k(n))$ and a system of contractive operators $f^\dagger_n:X_n\lra U_{k(n)}$ and $f^\ddagger_n:U_{k(n)}\lra X_n$ such that:
\begin{enumerate}
\item $\|f^\ddagger_n f^\dagger_n-{\bf 1}_{X_n}\|<2^{-n}$,
\item $\|f^\dagger_{n+1}|_{X_n}- f^\dagger_n\|<2^{-n}$,
\item $\|f^\ddagger_{n+1}|_{U_{k(n)}}- f^\ddagger_n\|<2^{-n}$.
\end{enumerate}
Since $\sum_n2^{-np}<\infty$ the remark closing the preceding Section shows that the pointwise limits of the sequences $(f^\dagger_n)$ and $(f^\ddagger_n)$  provide a contractive pair $\xymatrix{
X  \ar@<0.3ex>[r]
& \mathscr K_p\ar@<0.3ex>@{.>}[l] }
$
and completes the proof.

The required sequence is constructed by induction. We can assume $X_1=0$ and take $f^\dagger_1=0$ and $f^\ddagger_1=0$.

Now suppose that $f^\dagger_n:X_n\lra U_{k(n)}$ and $f^\ddagger_n:U_{k(n)}\lra X_n$ have been already constructed and let us see how to get $k(n+1)$ and the  maps
$f^\dagger_{n+1}:X_{n+1}\lra U_{k(n+1)}$ and $f^\ddagger_{n+1}:U_{k(n+1)}\lra X_{n+1}$.

We suggest the reader to take a pencil and paper for some scribbling.

$$
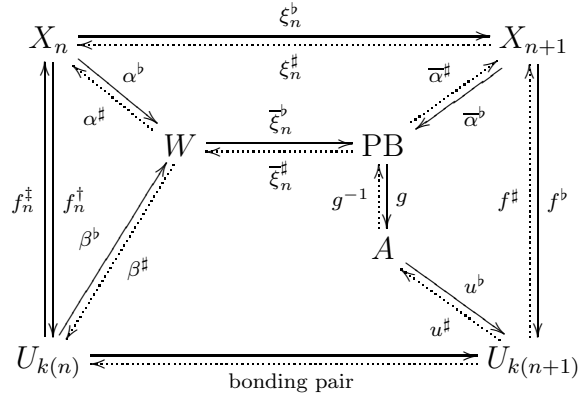
\xymatrix{
X_n \ar@<0.3ex>[rrrr]^{\xi^\flat_n} \ar@<0.3ex>[ddd]^{f^\dagger_n} \ar@<0.3ex>[rd]^{\alpha^\flat} && && X_{n+1} \ar@<0.3ex>@{.>}[llll]^{\xi^\sharp_n} \ar@<0.3ex>[ld]^{\overline{\alpha}^\flat}  \ar@<0.3ex>[ddd]^{f^\flat} \\
& W  \ar@<0.3ex>@{.>}[ul]^{\alpha^\sharp} \ar@<0.3ex>@{.>}[ddl]^{\beta^\sharp}  \ar@<0.3ex>[rr]^{\overline{\xi}^\flat_n} & & \PB \ar@<0.3ex>@{.>}[ur]^{\overline{\alpha}^\sharp} \ar@<0.3ex>@{.>}[ll]^{\overline{\xi}^\sharp_n} \ar@<0.3ex>[d]^g &
\\
&& & A  \ar@<0.3ex>[rd]^{u^\flat} \ar@<0.3ex>@{.>}[u]^{g^{-1}}&
\\
U_{k(n)} \ar@<0.3ex>[ruu]^{\beta^\flat} \ar@<0.3ex>[uuu]^{f^\ddagger_n}  \ar@<0.3ex>[rrrr]^{} && && U_{k(n+1)} \ar@<0.3ex>@{.>}[ul]^{u^\sharp} \ar@<0.3ex>@{.>}[llll]^{\text{bonding pair}} \ar@<0.3ex>@{.>}[uuu]^{f^\sharp}
}
$$

Set $\e =\|f^\ddagger_n\, f^\dagger_n-{\bf 1}_{X_n}\|<2^{-n}$ and reserve a small $\delta>0$ of room. The precise value of $\delta$ will be specified later.

\smallskip

$\bigstar$ First, we apply Lemma~\ref{lem:W} to $f^\dagger_n$ and $f^\ddagger_n$. In this way we obtain the space $W$ and the left triangle in the preceding diagram. Note that $\alpha=\langle \alpha^\flat, \alpha^\sharp\rangle$ and
$\beta=\langle \beta^\flat, \beta^\sharp\rangle$ are contractive pairs such that
$$
\|\beta^\flat f^\dagger_n-\alpha^\flat\|\leq \e,\quad\quad f^\dagger_n=\beta^\sharp \alpha^\flat,  \quad\quad f^\ddagger_n=\alpha^\sharp \beta^\flat.
$$

$\bigstar$ Then we amalgamate $\xi_n$ and $\alpha$ using Lemma~\ref{lem:PBwithpairs} which yields the upper commutative trapezoid.
\smallskip

$\bigstar$ Next we apply Lemma~\ref{lem:density} to the composition $\overline{\xi}_n\circ\beta$ (which is a contractive pair) thus obtaining a surjective $\delta$-isometry $g:\PB\lra A$ in such a way that the composition $\langle g,g^{-1}\rangle\circ \overline{\xi}_n\circ\beta$ turns out to be an allowed pair.
\smallskip

$\bigstar$ By the ``Fra\"{\i}ss\'e character'' of the sequence of pairs $(u_n)$
there must be some  $k(n+1)>k(n)$ and an allowed pair
$u: \xymatrix{
A  \ar@<0.3ex>[r]
& U_{k(n+1)} \ar@<0.3ex>@{.>}[l] }
$
such that $u\circ\langle g,g^{-1}\rangle\circ \overline{\xi}_n\circ\beta$ is the bonding pair $\xymatrix{
U_{k(n)} \ar@<0.3ex>[r]
& U_{k(n+1)} \ar@<0.3ex>@{.>}[l] }
$.
\smallskip

$\bigstar$ Now look at the pair $f=u\circ\langle g,g^{-1}\rangle\circ \overline{\alpha}$. Note that $f$ need not to be contractive as we only have the bound
$\|f\|\leq \| \langle g,g^{-1}\rangle\|\leq 1+\delta$.

\smallskip

Anyway, one has:

\begin{itemize}
\item[(4)] $f^\sharp  f^\flat={\bf 1}_{X_{n+1}}$,
\item[(5)] $\|f^\flat|_{X_n}-f^\dagger_n\|\leq (1+\delta)\e$,
\item[(6)] $f^\sharp|_{U_{k(n)}}=\xi^\flat_n  f^\ddagger_n$.
\end{itemize}

The first identity is trivial.
As for (4), note that $f^\flat|_{X_n}=u^\flat \, g\, \overline{\xi}^\flat_n\, \alpha^\flat$, hence
$$
\|f^\flat|_{X_n}-f^\dagger_n\|= \|u^\flat\,  g\, \overline{\xi}^\flat_n \,\alpha^\flat- \underbrace{u^\flat \, g\, \overline{\xi}^\flat_n\, \beta^\flat}_{\text{inclusion}}f^\dagger_n\|\leq \|g\|\| \beta^\flat \, f^\dagger_n-\alpha^\flat\|\leq (1+\delta)\e.
$$
To check (6) observe that the inclusion of $U_{k(n)}$ into  $U_{k(n+1)}$ can be written as $u^\flat\,g\,\overline{\xi}^\flat_n\,\beta^\flat$. Besides, $f^\sharp=\overline{\alpha}^\sharp\,g^{-1} u^\sharp$ so, recalling that $\overline{\alpha}^\sharp\, \overline{\xi}^\flat_n=\xi^\flat\, \alpha^\sharp$, we have
$$
f^\sharp|_{U_{k(n)}}=  \overline{\alpha}^\sharp\, g^{-1} u^\sharp\, u^\flat\, g\, \overline{\xi}^\flat_n\,\beta^\flat
=  \overline{\alpha}^\sharp\, \overline{\xi}^\flat_n\,\beta^\flat = \xi^\flat_n\alpha^\sharp\beta^\flat= \xi^\flat_n f^\ddagger_n.
$$
A final touch just to render the maps contractive. Set
$$
f^\dagger_{n+1}=\frac{f^\flat}{1+\delta}\quad\quad\text{and}\quad\quad f^\ddagger_{n+1}=\frac{f^\sharp}{1+\delta}.
$$
Then $f^\ddagger_{n+1} f^\dagger_{n+1}=(1+\delta)^{-2}{\bf 1}_{X_{n+1}}$, hence using (4),
$$
\|f^\ddagger_{n+1}  f^\dagger_{n+1}-{\bf 1}_{X_{n+1}}\|\leq 1-\frac{1}{(1+\delta)^2}<\frac{1}{2^{n+1}}
$$
for $\delta$ sufficiently small.

Also, from (5) and (6) we get
$$
\begin{aligned}
\|f^\dagger_{n+1}|_{X_n}-f^\dagger_{n}\|&\leq\left( \|f^\dagger_{n+1}-f^\flat\|^p+ \|f^\flat|_{X_n}-f^\dagger_{n}\|^p	\right)^{1/p}\leq (\delta^p+(1+\delta)^p\e^p)^{1/p}<2^{-n},\\
\|f^\ddagger_{n+1}-\xi^\flat_n  f^\ddagger_{n}\|&= \|f^\ddagger_{n+1}-f^\sharp\|\leq \delta<2^{-n},
\end{aligned}
$$
for $\delta$ sufficiently small.
\end{proof}

\subsection{Uniqueness} As we already mentioned Kalton observed in \cite{kaltuni} that for each $0<p<1$ there is a complementably universal $p$-Banach space for the BAP. For $p=1$ an earlier example, sprung from two different lines of reseach carried out independently by Kadec and Pelczy\'nski, was known. It turs out that ``our'' space $\mathscr K_p$ is just a ``renorming'' of Kalton's:

\begin{proposition}
A separable $p$-Banach space is complementably universal for the BAP if and only if it is linearly isomorphic to $\mathscr K_p$.
\end{proposition}

\begin{proof}
The space $\mathscr K_p$  is complementably universal for the BAP. Indeed, every separable $p$-Banach space with the BAP is complemented in one with a basis, which clearly has a 1FDD (actually ``one-dimensional'') under an equivalent $p$-norm.

As for the converse, just observe that there is only a separable $p$-Banach space, complementably universal for the BAP, up to linear isomorphisms, by Pe\l czy\'nski decomposition method.
\end{proof}

It is difficult to imagine a space peskier than $\mathscr K_p$: indeed, the following spaces are all isomorphic to $\mathscr K_p$:
\begin{itemize}
\item Direct sums $\mathscr K_p\oplus X$, when $X$ is a separable $p$-Banach space with the BAP.
\item Spaces of $\mathscr K_p$-valued sequences $X(\mathbb N, \mathscr K_p)$, when $X$ is a $p$-Banach sequence space, in particular $\ell_q(\mathbb N, \mathscr K_p)$ for $0<p\leq q<\infty$ and   $c_0(\mathbb N, \mathscr K_p)$.
\item The $p$-Banach space tensor products $X\hat{\otimes}_p \mathscr K_p$ when when $X$ is a separable $p$-Banach space with the BAP; cf. \cite{turpin}.
\item The $p$-convex envelope of $ \mathscr K_q$ for $0<q<p$.
\end{itemize}

In contrast the space $\mathscr K_p\oplus_p L_p$ is of almost universal complemented disposition and not isomorphic to $\mathscr K_p$ if $p<1$.

We now address the ``isometric uniqueness'' of $\mathscr K_p$ and its ``rotational'' properties. Our main result in this line is Theorem~\ref{th:iso-uniq}, which generalizes Theorem 6.3 in \cite{GW}. Our route to the proof is slightly different from that of \cite{GW} because
we are not sure that  \cite[Lemma 6.2]{GW} is true as stated.

The following ``stability'' result is  interesting in its own right:

\begin{proposition}
Let $X$ be a $p$-Banach space with a 1FDD and satisfying $[\Game]$.
Let $f^\dagger:E\lra X$ and $f^\ddagger:X\lra E$ be contractive operators such that $\|f^\ddagger f^\dagger-{\bf 1}_E\|<\e$, where $E$ is a finite dimensional $p$-Banach and $\e>0$.

Then  there is an isometry $f^\flat: E\lra X$ whose range is 1-complemented and such that  $ \|f^\dagger - f^\flat\| < \e$.
\end{proposition}

\begin{proof}
We fix a 1FDD $(X_n)$ of $X$ and we denote by $\xi_n: \xymatrix{
X_n  \ar@<0.3ex>[r]
& X \ar@<0.3ex>@{.>}[l] }
$
and $\xi_{(n,k)}: \xymatrix{
X_n  \ar@<0.3ex>[r]
& X_k \ar@<0.3ex>@{.>}[l] }
$
the corresponding pairs of operators.

We also fix a sequence $(\e_n)_{n\geq 0}$ of positive numbers with $\e_1<\e$
 such that $\|f^\ddagger f^\dagger-{\bf 1}_E\|<\e_1$ and
 $\sum_{n\geq0}\e_n^p<\e^p$. Note that we must first choose $\e_1$ and then the other $\e_n$'s.

 By using a small automorphism of $X$ one can obtain $n(0)$ and contractive operators
$f_0^\dagger:E\lra X_{n(0)}$ and $f_0^\ddagger:X_{n(0)}\lra E$ so that
$$\|f^\dagger - f_0^\dagger\|< \e_0\quad\quad\text{and}\quad\quad\|f_0^\ddagger f_0^\dagger-{\bf 1}_E\|\leq \e_1.$$
Let us apply Lemma~\ref{lem:W} to $f_0^\dagger, f_0^\ddagger$ and $\e_1$ to obtain the diagram
$$
\xymatrix{
E  \ar@<0.4ex>[rrrd]^{\alpha^\flat}   \ar@<0.4ex>[dd]^{f_0^\dagger} \\
&&& W  \ar@<0.3ex>@{.>}[ulll]^{\alpha^\sharp}   \ar@<0.3ex>@{.>}[dlll]^{\beta^\sharp} \\
X_{n(0)} \ar@<0.4ex>[rrru]^{\beta^\flat}   \ar@<0.4ex>[uu]^{f_0^\ddagger}
}
$$
where $\alpha$ and $\beta$ are contractive pairs and
$$
f_0^\dagger=\beta^\sharp\alpha^\flat, \quad f_0^\ddagger=\alpha^\sharp\beta^\flat, \quad \|\alpha^\flat - \beta^\flat f_0^\ddagger\|\leq \e_1.
$$
As $X$ has property $[\Game]$, after normalizing a suitable almost-isometry $W\lra X$ and the corresponding projection, one obtains $n(1)>n(0)$ and two contractive operators $\gamma^\dagger: W\lra X_{n(1)}$
 and $\gamma^\ddagger: X_{n(1)}\lra W$ such that
$$\|\gamma^\ddagger\gamma^\dagger-{\bf 1}_W\|<\e_2\quad\quad\text{and}\quad\quad \|\gamma^\dagger\beta^\flat- \xi_{(n(0),n(1))}^\flat\|<\e_2.$$
Letting $f_1^\dagger=  \gamma^\dagger \alpha^\flat $ and $f_1^\ddagger= \alpha^\sharp \gamma^\ddagger$ we have $\|f_1^\ddagger f_1^\dagger-{\bf 1}_E\|<\e_2$ and
$$
\|f_1^\dagger-f_0^\dagger\|^p
=
\| \gamma^\dagger \alpha^\flat- \gamma^\dagger \beta^\flat f_0^\dagger +  \gamma^\dagger \beta^\flat f_0^\dagger - f_0^\dagger\|^p\leq
\|  \alpha^\flat-  \beta^\flat f_0^\dagger\|^p + \| \gamma^\dagger \beta - \xi_{(n(0),n(1))}^\flat\|^p< \e_1^p+\e_2^p.
$$
Continuing in this way one obtains an increasing sequence $(n(k))_{k\geq 0}$ and contractive operators
$f_k^\dagger:E\lra X_{n(k)}$ and $f_k^\ddagger:X_{n(k)}\lra E$ such that
\begin{itemize}
\item $\|f_k^\ddagger f_k^\dagger-{\bf 1}_E\|\leq \e_{k+1}$.
\item $\|f_{k+1}^\dagger-f_k^\dagger\|< \left(\e_{k+1}^p+\e_{k+2}^p\right)^{1/p}$,
\end{itemize}
The second estimate implies that $(f_k^\dagger)_k$ is a Cauchy sequence in $L(E,X)$ since
$$
\|f_{k+m}^\dagger - f_k^\dagger\|\leq
\left(\sum_{i=0}^{m-1} \|f_{k+i+1}^\dagger - f_{k+i}^\dagger\|^p 	\right)^{1/p}
\leq \left(\sum_{i=0}^{m-1} \e_{k+i+2}^p +  \e_{k+i+1}^p	\right)^{1/p}.
$$
And then the first one implies that the ``double sequence'' $(f_k^\ddagger f_n^\dagger)_{k,n}$ converges to the identity of $E$ in the sense that for every $\delta>0$ there is $m$ such that $\|f_k^\ddagger f_n^\dagger-{\bf 1}_E\|< \delta$
whenever $k,n\geq m$.

Define $f^\flat: E\lra X$ as the limit of the sequence  $(f_k^\dagger)_k$, that is
$$
f^\flat(y)=\lim_k f_k^\dagger(y).
$$
To obtain a suitable projection along $f^\flat$ we can use the local compactness of $E$: take a nontrivial ultrafilter $\mathscr U$ on the integers, put
$$
f^\sharp(x)=\lim_{{\mathscr U}(n)} f_n^\ddagger(x)
$$
for $x\in \bigcup_k X_k$ and extend by continuity to all of $X$. It clear that $f^\flat$ and $f^\sharp$ are contractive. Finally, given $y\in E$, one has
$$
f^\sharp f^\flat(y)= \lim_{{\mathscr U}(n)} f_n^\ddagger\left( f^\flat(y)\right)
=
\lim_{{\mathscr U}(n)} f_n^\ddagger\left(\lim_{k\to\infty} f_k^\dagger(y)\right)
=
\lim_{{\mathscr U}(n)}  \left(\lim_{k\to\infty} f_n^\ddagger f_k^\dagger(y)\right)
=
\lim_{k,n} f_n^\ddagger f_k^\dagger(y)= y.
$$
This shows that $f^\flat$ is an isometry whose range is 1-complemented in $X$.
\end{proof}

\begin{theorem}\label{th:iso-uniq}
Suppose $X$ and $Y$ are $p$-Banach spaces with 1FDDs and satisfying $[\Game]$.
Let $f_0:A\lra B$ be a surjective isometry, where $A$ is a 1-complemented subspace of $X$
and $B$ is a 1-complemented subspace of $Y$. Then for every $\e>0$ there is a surjective isometry $f: X\lra Y$ such that $\|f|_A-f_0\|<\e$.
\end{theorem}

\begin{proof}
The proof is a typical back-and-forth argument oiled by the preceding Proposition.
Fix a sequence of positive real numbers $(\e_n)_{n\geq 0}$, such that $\sum_{n\geq 0}\e_n^p<\e^p$. Let $(X_n)$ and $(Y_n)$ be increasing sequences of finite dimensional 1-complemented subspaces of $X$ and $Y$, respectively, with dense union.

Take $A_1=A+X_1$. Then $f_0^{-1}$ embeds isometrically $B$ into $A_1$, as a 1-complemented subspace. As $Y$ has property $[\Game]$, for each $\delta>0$ we can find
a $\delta$-isometry $f_{1/2}: A_1\lra Y$ whose range is $(1+\delta)$-complemented extending the inclusion of $B$.

Applying the preceding Lemma to $f_{1/2}$ with $\delta$ small enough we obtain an isometry $f_1:A_1\lra Y$ with 1-complemented range such that $\|f_1(f_0^{-1}(y))-y\|<\e_1\|y\|$ for all nonzero $y\in B$.

Set $B_1=f_1[A_1]+B+Y_1$ and apply the same argument to obtain an isometry $g_1:B_1\lra X$ with 1-complemented range with $\|g_1(f_1(x))-x\|<\e_1\|x\|$ for all nonzero $x\in A_1$.

Now set $A_2=g_1[B_1]+A_1+X_2$ and let $f_2:A_2\lra Y$ be an isometry with 1-complemented range such that $\|f_2(g_1(y))-y\|<\e_2\|y\|$ for all nonzero $y\in B_1$ and so on.

Continuing in this way we obtain two increasing sequences $(A_n)_{n\geq 0}$ and $(B_n)_{n\geq 0}$ of finite dimensional 1-complemented subspaces of $X$ and $Y$, respectively, with dense union, where $A_0=A$ and $B_0=B$ together with isometries $f_n:A_n\lra B_n$ and $g_n:B_n\lra A_{n+1}$ satisfying
\begin{itemize}
\item[(1)] $\|g_n(f_n(x))-x\|<\e_n\|x\|$ for all nonzero $x\in A_n$;
\item[(2)] $\|f_{n+1}(g_n(y))-y\|<\e_n\|y\|$ for all nonzero $y\in B_n$,
\end{itemize}
where $g_0=f_0^{-1}$. The situation is illustrated in the following (``almost commutative'') diagram
$$
\xymatrix{
A \ar[r]\ar[d]^{f_0}  & A_1 \ar[dr]^{f_1} \ar[rr] && A_2 \ar[dr]^{f_2} \ar[rr] & &\cdots\\
B \ar[rr]\ar[ur]_{g_0} &  & B_1  \ar[rr] \ar[ur]_{g_1} && B_2 \ar[ur]_{g_2} \ar[rr] & &\cdots\\
}
$$
Let us define an operator $f:X\lra Y$ as follows. For $x\in A_k$ put
$$
f(x)=\lim_{n\geq k} f_n(x).
$$
The definition makes sense since $(f_n(x))_{n\geq k}$ is a Cauchy sequence. Indeed, for $x\in A_n$ we have
$$
\begin{aligned}
\|f_{n+1}(x)-f_n(x)\|&\leq \left(\|f_{n+1}(x)- f_{n+1}(g_n(f_n(x)))\|^p+ \| f_{n+1}(g_n(f_n(x))) -f_n(x)\|^p\right)^{1/p}\\
&\leq  \left(\|f_{n+1}\|^p\|x- g_n(f_n(x))\|^p+ \e_n^p \|f_n(x)\|^p\right)^{1/p}\\
&\leq 2^{1/p}\e_n\|x\|.
\end{aligned}
$$
Since $\sum_{n\geq 0}\e_n^p$ is finite we see that $f$ is correctly defined on $\bigcup_nA_n$ and so it extends to a contractive operator on $X$ that we call again $f$.

Besides, for  $x$ in $A=A_0$ one has
$$
\|f(x)-f_0(x)\|\leq \left( \sum_{n\geq 0}\|f_{n+1}(x)-f_n(x)\|^p 	\right)^{1/p}\leq  \left(\sum_{n\geq 0}2\e_n^p\right)^{1/p}\leq 2^{1/p}\e.
$$
Proceeding analogously with the sequence $(g_n)$ one obtains a contractive operator $g:Y\lra X$ such that
$$
g(y)=\lim_{n\geq k}g_n(y)\quad\quad(y\in B_k).
$$
Moreover, it is clear from (1) and (2) that $gf={\bf 1}_X$ and $f g={\bf 1}_Y$, which completes the proof.
\end{proof}

Let us draw some consequences of Theorem~\ref{th:iso-uniq}.
First, any $p$-Banach space with a 1FDDs and property $[\Game]$ is isometric to $\mathscr K_p$ and, therefore, of almost universal complemented disposition.

The 1FDD requirement is quite natural in our setting, since it corresponds to ``separability'' in the category of contractive pairs. We refer the reader to \cite{CM-ACUD} for a more complete classification of Banach spaces of (almost) universal complemented disposition. There, by using an ``enveloping'' technique, it is shown that every Banach space whose dual is separable is isometric to a 1-complemented subspace of a separable Banach space of almost universal complemented disposition. Thus, starting with a separable and reflexive Banach space lacking the AP one obtains Banach spaces of (almost) universal complemented disposition different from  $\mathscr K_1$.

Second, the Banach space $\mathscr K_1$ (that is, Garbuli\'nska renorming of Kadec/Pe\l czy\'nski/Woj\-taszczyk space) is ``almost isotropic'': given $x,y\in\mathscr K_1$ with $\|x\|=\|y\|=1$ and $\e>0$ there is an isometric automorphism $f$ of $\mathscr K_1$ such that $\|y-f(x)\|\leq\e$. This clearly follows from Theorem~\ref{th:iso-uniq} and the fact that all lines are   1-complemented in all Banach spaces, by Hahn-Banach. That cannot we achieved for $\e=0$ since the unit sphere of $\mathscr K_1$ contains (many) points where the norm is smooth and points where it is not (think of an isometric copy of, say, $\ell_\infty^2$) and note that a surjective isometry must preserve both classes.

In contrast, there is no equivalent $p$-norm rendering $\mathscr K_p$ almost isotropic when $p<1$. For if $X$ is almost isotropic and linearly isomorphic with $\mathscr K_p$, then the functional
$$
|x|=\|x\|+\sup_{\|x^*\|\leq 1}|x^*(x)|
$$
is another $p$-norm that is preserved by every isometry for the original $p$-norm of $X$. It quickly follows (cf. \cite[Theorem~3.3]{maximal} for the complete argument) that $|x|=2\|x\|$ for all $x\in X$ and so $\|x\|=\sup_{\|x^*\|\leq 1}|x^*(x)|$, that is, $X$ is locally convex, which is not the case.


\subsection{Wheeling around $\e$}
It is clear that moving the number $\e$ from here to there in the definitions opening  Section~3 one obtains other variants that are more or less equivalent to these appearing in the text.

Actually the version of property $[\Game]$ and the definition of a space of almost universal complemented disposition that we have used here do not agree with those of \cite{GW} and \cite{CM-ACUD}. The following simple remark shows that $[\Game]$ is equivalent to Garbuli\'nska's property (E) of \cite{GW} and that Definition~\ref{def:aucd} is equivalent to Definition
3.1 in \cite{CM-ACUD}.

\begin{lemma}[Correcting a defective pair]\label{defec}
Let $f^\dagger:E\lra F$ and $f^\ddagger:F\lra E$ be operators such that $\|f^\ddagger\, f^\dagger-{\bf 1}_E\|\leq \varepsilon$, where $\varepsilon<1$. Then there is an automorphism $a$ of $E$ such that
\begin{itemize}
\item $\|a-{\bf 1}_E\|\leq \e(1-\varepsilon^p)^{-1/p}$,
\item $\|a\|\leq (1-\varepsilon^p)^{-1/p}$,
\item $\|a^{-1}\|\leq (1+\varepsilon^p)^{1/p}$,
\item $a f^\ddagger$ is a projection along $f^\dagger$, that is, $a f^\ddagger f^\dagger ={\bf 1}_E$.
\end{itemize}
\end{lemma}

\begin{proof}
Take $a=\sum_{n\geq 0}({\bf 1}_E- f^\ddagger f^\dagger)^n$ and check.
\end{proof}

 However, ``an $\e$
of room'' is necessary to stay in the separable world, as the following example shows.

\begin{proposition}
Let $X$ be a $p$-Banach space having the following property: if $F$ is a $p$-normed of dimension 3 (or less), $E$ is a 1-complemented subspace of $F$ and  $u:E\lra X$ is an isometry with 1-complemented range, then there is an isometry $v:F\lra X$ with 1-complemented range such that $v|_{E}=u$. Then the density character of $X$ is at least the continuum.
\end{proposition}

\begin{proof}
The proof uses an idea of Haydon, taken from \cite{CGK}.
Let us do it in the real case.
The hypothesis implies that $X$ contains a complemented copy of each $p$-normed space of dimension up to 3. Let $E$ be the Euclidean plane and $S$ the unit sphere of $E$.

For each $u\in S$, we consider the following $p$-norm on $E\times\R$:
$$
|(x,t)|_u= \max \left( \|x\|_2, \|(\langle x|u\rangle,t)\|_p	\right)
$$
and let $F_u$ denote the resulting 3-dimensional space.
(The unit ball of $F_u$ is the intersection of a ``vertical'' right cylinder and a ``horizontal'' right prism whose basis is the two dimensional $p$-ball with ``peaks'' at $(0,1)$ and $(u,0)$.)

 Note that
$\|x,0\|_u=\|x\|_2$ (so $E$ is isometric to a subspace of $F_u$) and that
$|(x,t)|_u\geq \|x\|_2$ for each $(x,t)\in F_u$ (so the obvious projection is contractive).

Now we consider $E$ as a 1-complemented subspace of $X$ and let us assume that for every $u\in S$ one can find an isometry $f_u:F_u\lra X$ such that $f_u(x,0)=x$.

Clearly, $f_u$ must have the form $f_u(x,t)=x+te_u$, for some fixed $e_u$ in the unit sphere of $X$.

Now let $S_+$ be the ``positive part'' of $S$, so that $0<\langle u,v\rangle<1$ if $u,v\in S_+$ are different. We claim that $\|e_u-e_v\|=1$ for $u,v\in S_+$ unless $u=v$.

To proceed, pick any $\lambda>0$. We have
$$
\|e_u-e_v\|^p\geq \|e_u+\lambda u\|^p - \|e_v +\lambda u\|^p= |(\lambda u,1)|_u^p-
 |(\lambda u,1)|_v^p.
$$
But $|(\lambda u,1)|_u^p=1+\lambda^p$, while for large $\lambda$,
$$
|(\lambda u,1)|_v^p= \max\left(\lambda^p, 1+\lambda^p\langle u|v\rangle^p\right)= \lambda^p.
$$
Hence the density character of $X$ is, at least, the cardinality of $S_+$.
\end{proof}

\section{Some questions to end}

We take our leave of the reader with a list of questions that we encountered along the research that we have reproted in this paper. Some of them might be very silly exercises, but we don't know the answer.

\begin{itemize}

\item The proof of Corollary~\ref{cor:theend}
shows that $ Y\oplus\left(\co^{(p)}(X)\oplus_\mho X\right)$ is isomorphic to
$\co^{(p)}(X)\oplus \ell_p$. However we don't know if, with the notations of that Corollary, $Y$ is isomorphic to $\co^{(p)}X$ or $\co^{(p)}(X)\oplus_\mho X$ to $\ell_p$.

\item Let  $Q: \ell_p\lra X$ be a quotient map. Does $\ker Q$ have a basis if $X$ has a basis? And $\co^{(p)}X$?

\item Does every quotient $\ell_p\lra X$ have a continuous selection? What if $X=\mathscr K_p$?

\item If $X$ is a $p$-Banach space with separating dual it is always possible to find $x^*\in X^*$ attaining the norm at some point of the unit ball of $X$?

\item If $K$ is a compact metric space, does the space $C(K,\mathscr K_p)$ have the BAP? The answer is affirmative if $K$ is zero dimensional, but it seems to be unknown even if $K$ is the unit interval.

\item Does Kadec' space in \cite{kade} have property $[\Game]$ in its own norm?

\item Do the isometry groups of the spaces $\mathscr K_p$ have any amenability property in, say, the strong operator topology? Does the space  $\mathscr K_p$ admit a surjective isometry of the form $T={\bf 1}_{\mathscr K_p}+F$, where $F$ is a finite rank operator? Is the isometry group of the real spaces $\mathscr K_p$ discrete in the norm topology?

\end{itemize}

\section*{Appendix: pull-back}
In the papers \cite{CGK, CM-ACUD, GW} ``amalgamations'' are invariably constructed by means of push-outs. Here we take the ``dual view point'' which is perhaps more direct since it does not depend on the ambient category nor uses quotients. So let us have a little chat about the pull-back construction for quasi Banach spaces.

We first explain what we need and then how one can manage to get it.

Let $\alpha:A\lra E$ and $\beta:B\lra E$ be operators acting between $p$-Banach spaces.
What we need is another $p$-Banach space $\PB$, together with contractive operators $\overline{\alpha}$ and  $\overline{\beta}$ making commutative the square
\begin{equation}\label{pb-dia}
\begin{CD}
\PB@>\overline{\beta}>> A\\
@V \overline{\alpha} VV @VV \alpha V\\
B @> \beta >> E
\end{CD}
\end{equation}
Moreover, and this is the crucial point, the square has to be ``minimally commutative'' in the sense that for every couple of operators
$\beta': C\lra A$ and $\alpha': C\lra B$ satisfying $\alpha \beta'=\beta\alpha'$,
there is a unique operator $\gamma:C\lra\PB$ such that:
\begin{itemize}
\item $\alpha'=\overline{\alpha} \gamma$ and
$\beta'=\overline{\beta} \gamma$ and ${''}\!\beta={'}\!\beta \gamma$,
\item $\|\gamma\|\leq \max\big{(} \|\alpha'\|, \|\beta'\|	\big{)}$.
\end{itemize}
This universal property can be visualized in the commutative
diagram
\begin{equation}\label{diag:PB}
\xymatrix{C \ar[dr]^{\gamma} \ar@/^1pc/[rrd]^{\beta'} \ar@/_1pc/[rdd]_{\alpha'}\\
 & \PB \ar[r]^{\overline{\beta}} \ar[d]_{\overline{\alpha}} & A \ar[d]^{\alpha} \\
 & B\ar[r]^{\beta} & E \\
}
\end{equation}
hence the term ``pull-back''.

In particular the space $\PB$ is unique, up to isometries.
Having said this, let us describe a ``concrete'' representation.

The pull-back space is $\PB=\PB(\alpha,\beta)=\{(a,b)\in A\oplus_\infty B: \alpha(a)=\beta(b) \}$.
The arrows under bars are the restriction of the projections onto the corresponding factor.
These work as required since if
$\beta': C\lra A$ and $\alpha': C\lra B$ satisfy $\alpha \beta'=\beta\alpha'$,
then we can set $\gamma(c)=(\beta'(c),\alpha'(c))$.

It is important to realize that if $\alpha$ and $\beta$ are rational maps,
then $\PB$ has a basis of rational vectors of $A\times B=\K^{n+m}$ and therefore it can be regarded as $\K^k$ equipped with a $p$-norm that has to be allowed if those of $A$ and $B$ are, by conditions (2) and (3) in the definition of Section~\ref{sec:allowed}.

The ``projections''  $\overline{\alpha}$ and $\overline{\beta}$ are then rational maps
and the universal property of Diagram~\ref{diag:PB} ``preserves'' rational maps in the sense that if $\beta'$ and $\alpha'$ are rational maps, then so is $\gamma$.

 Please keep this in mind when reading the Proof of Lemma~\ref{lem:PBwithpairs}.

\end{document}